\documentclass[english]{amsart}

\usepackage{esint}
\usepackage[svgnames]{xcolor} 
\usepackage{pgf,tikz}

\usepackage{dsfont}
\usepackage{url}
\usepackage[utf8]{inputenc}
\usepackage[T1]{fontenc}
\usepackage{lmodern}
\usepackage{babel}
\usepackage{mathtools}  
\usepackage{amssymb}
\usepackage{lipsum}
\usepackage{mathrsfs}
\usepackage{color}
\usepackage[skip=2.2pt plus 1pt, indent=12pt]{parskip}
\usepackage{stmaryrd}
\usepackage{soul}
\usepackage{ulem}

\usepackage{amsmath, amsthm, amssymb, mathrsfs}

\usepackage{booktabs}
\usepackage{caption}
\usepackage{microtype}

\usepackage{graphicx}
\newtheorem{definition}{Definition}[section] % 从subsection层次开始排序,例如1.1.1,显示为“定义 1.1.1”
\newtheorem{theorem}{Theorem}[section] % 从section层次开始排序,例如1.1
\newtheorem{lemma}[theorem]{Lemma} % 与theorem共享编号,例如现在是theorem的编号为1.1,下一个lemma的编号就是1.2
\newtheorem{proposition}[theorem]{Proposition} % 与theorem共享编号
\theoremstyle{remark}
\newtheorem{remark}[theorem]{Remark}

\usepackage{url}
\usepackage{hyperref} % Elsevier模板通常兼容hyperref
\usepackage{color}
\title{On Generalized Barron Spaces for Shallow Neural Networks} %% Article title

\author{Shuai Lu and Leqi Zhang} %% Author name
\address{School of Mathematical Sciences, Fudan University, 220 Handan Road, Shanghai 200433, China. Corresponding author: S. Lu}
\email{slu@fudan.edu.cn (S. Lu); 23110180058@m.fudan.edu.cn (L. Zhang) }

\begin{document}

\begin{abstract}
Classical Barron spaces are function spaces specifically designed for shallow neural networks mostly with ReLU, $\mathrm{ReLU}^k$ (RePU) or Lipschitz continuous activation functions. In the present work, we introduce a generalized Barron space \( B_{\sigma}^{\varphi} \) for shallow neural networks with a generic activation function possessing certain smoothness properties. The subscript \( \sigma \) denotes the activation function, while the superscript \( \varphi \) controls the smoothness of the generalized Barron spaces, defined via a \( \varphi \)-weighted integral norm imposed on neural network parameter measures.  
Under certain assumptions on \( \varphi \) and \( \sigma \), we show that \( B_{\sigma}^{\varphi} \) can be continuously embedded into Sobolev spaces. We also explore the relationships among various function spaces for shallow neural networks, demonstrating that our definition encompasses most conventional ones.  
As applications of the proposed generalized Barron spaces, we derive approximation rates within these spaces and establish error bounds for numerical differentiation with regularization penalized by the newly introduced generalized Barron norm. Numerical examples confirm that the proposed spaces allow the construction of neural networks with varying degrees of smoothness while using the same activation function.
\end{abstract}



%% Keywords


\maketitle

\section{Introduction}
The study of shallow neural networks dates back several decades, originating in the 1940s. In early contributions, \cite{MP_model} introduced the M-P neuron model, which provided the first mathematical simulation of biological neurons and established an early theoretical foundation for neural computation. Subsequently, \cite{perceptron} developed the perceptron, which is a learnable single-layer neural network capable of linear classification, and ignited the first wave of enthusiasm in artificial intelligence. Since then, the term shallow neural network has primarily referred to architectures such as the single-layer perceptron and related early models. The significance of these networks lies in their role in promoting efficient approaches to AI, establishing core components such as neurons and weight learning, and offering valuable insights that shaped the subsequent development of neural networks and high-dimensional numerical simulation \cite{MR4420580, Bach2017, Cybenko89, DeVore2021, Hornik91, deep_learning, GKNV2022, analytic_app, Yarotsky2017, CNN_app}.
Mathematically, a shallow neural network can be represented as
\begin{align}\label{eq_shallowNN}
f_n(x) = \frac{1}{n}\sum_{i=1}^{n} a_i \sigma(w_i \cdot x + b_i),
\end{align}
where \( n \) denotes the number of neurons; \( a_i \) and \( w_i \) are the outer and inner layer weights, respectively; and \( b_i \) is the bias in the inner layer. The function \( \sigma \) is a nonlinear activation function, which may take forms such as ReLU, $\mathrm{ReLU}^k$, or Sigmoid.
Recently, to advance the mathematical understanding of shallow neural networks, several function spaces have been investigated, including spectral Barron spaces \cite{Barron1993}, (extended) Barron spaces \cite{E_Barronspace, Li_derivativeapprox}, variational spaces \cite{variation_space}, and Radon-BV spaces \cite{RadonBV}. Within these frameworks and the classical Sobolev setting, quantitative universal approximation results have been established, primarily for activation functions such as ReLU, $\mathrm{ReLU}^k$, and sigmoid; see, for instance, \cite{abdeljawad2024weighted, Barron1993, BN03, E_Barronspace, jiao_Deep_2023, klusowski2018approximation, Li_derivativeapprox, Liao_SJMDS25, variation_space, siegel2020approximation, Xu_sharpbound, Xu20} and references therein.

%we aim to introduce a novel and generalized class of Barron spaces for shallow neural networks, unifying the concepts of spectral Barron spaces and (extended) Barron spaces. We refer to this unified framework as Barron-type spaces. 

In this work, we focus specifically on two of the aforementioned spaces: the spectral Barron space and the (extended) Barron space, collectively referred to as Barron-type spaces. These function spaces are designed for the analysis of shallow neural networks. The first concept originated from the seminal work of \cite{Barron1993}, which introduced a class of functions that can be approximated by a shallow neural network with \( n \) hidden units, for instance as in (\ref{eq_shallowNN}), to an accuracy of order \( O(n^{-1/2}) \), thereby avoiding the curse of dimensionality. This perspective has spurred extensive follow-up research on high-dimensional approximation using neural networks.  
We now provide a brief introduction to these function spaces. The space is introduced in \cite{Barron1993} within a semi-norm setting, and is now extended as the spectral Barron space in \cite{ChenLuSpectral2023, choulli2026, klusowski2018approximation, lu2026, Xu_sharpbound, Xu20}. In current manuscript, this space is denoted by \(\mathscr{B}^s\). For any \(s \ge 0\) and a bounded domain \(\Omega \subset \mathbb{R}^d\), it is defined as  
\[
\mathscr{B}^s := \bigl\{ f = f_e|_{\Omega} \in C_b(\Omega) \;\big|\; (1+|\xi|)^s \hat{f_e} \in L^1(\mathbb{R}^d) \bigr\},
\]
where \(C_b(\Omega)\) denotes the space of bounded continuous functions on \(\Omega\), and \(f_e\) is a bounded continuous extension of \(f\) from \(\Omega\) to \(\mathbb{R}^d\). 
By denoting the following value 
\[
\|f\|_{\mathscr{B}^s_{f_e}} := \int_{\mathbb{R}^d} (1+|\xi|)^s |\hat{f_e}(\xi)| \, \mathrm{d}\xi,
\]
the spectral Barron norm is defined by  
\[
\|f\|_{\mathscr{B}^s} := \inf_{f_e|_\Omega = f} \|f\|_{\mathscr{B}^s_{f_e}}.
\]

Building upon the foundational work in \cite{Barron1993}, subsequent studies on shallow neural networks have characterized Barron-type spaces using parameter norms. In modern contexts, it is common to analyze networks equipped with ReLU or $\mathrm{ReLU}^k$ activation functions, where model complexity is evaluated through the expected \(\ell_1\) norm of the parameters. However, this framework relies crucially on the property of positive homogeneity. For non-homogeneous activation functions, such as sigmoid, tanh, or GELU, a straightforward reliance on parameter norms may become degenerate and fail to yield a well-defined Banach space. We also briefly provide some necessary background material, which can be found in \cite{E_Barronspace, Li_derivativeapprox} and will serve as the starting point of our current work. In particular, we recall the flow-induced function spaces introduced in \cite{E_Barronspace}, referred to as Barron spaces, described as follows. For \(k=1\), assume that a function \( f \) admits the representation
\[
f(x)=\int_{\mathbb{R}^{d+1}} \mathrm{ReLU}^k(w\cdot x+b) \, \mathrm{d}\rho(w,b),
\]
with
\[
\Vert f\Vert_{B^k_{\rho}}:=\int_{\mathbb{R}^{d+1}} \mathrm{ReLU}^k\bigl(\Vert w\Vert_1+\vert b\vert\bigr) \, \mathrm{d}\lvert\rho\rvert < \infty,
\]
for some \( \rho \in\mathcal{M}(\mathbb{R}^{d+1}) \). All such functions \( f \) constitute a space denoted by \( B^k \), and the corresponding Barron norm is defined as
\[
\Vert f\Vert_{B^k} := \inf_{\rho \in \mathcal{G}_f} \Vert f\Vert_{B^k_{\rho}},
\]
in which $\mathcal{G}_f = \left\{ \rho |
 f(x)=\int_{\mathbb{R}^{d+1}} \mathrm{ReLU}^k(w\cdot x+b) \, \mathrm{d}\rho, \Vert f\Vert_{B^k_{\rho}}<\infty,  \rho \in\mathcal{M}(\mathbb{R}^{d+1}) \right\}$. For results concerning all finite integers \( k \geq 1 \), which lead to the definition of the extended Barron spaces, we refer to \cite{Li_derivativeapprox}.

The above definition of the (extended) Barron norm and (extended) Barron spaces relies heavily on the homogeneity of the \(\mathrm{ReLU}^k\) activation function.  
It is worth highlighting that \cite{E_Barronspace} also introduced a Barron norm applicable to neural networks with Lipschitz-type activation functions and we briefly recall below.  
Suppose there exists a Lipschitz continuous activation function \(\sigma\), and a function \(f\) admits the following integral representation:
\[
f(x)=\int_{\mathbb{R}^{d+1}} \sigma(w\cdot x+b) \, \mathrm{d}\rho(w,b),
\]
with  
\[
\Vert f\Vert_{B_{\sigma,\rho}}:=\int_{\mathbb{R}^{d+1}} \bigl(1+\Vert w\Vert_1+|b|\bigr) \, \mathrm{d}|\rho|<\infty
\]
for some \(\rho \in\mathcal{M}(\mathbb{R}^{d+1})\).  
Then all such functions \(f\) constitute a space denoted by \(B_{\sigma}\), and the corresponding Barron norm is defined as  
\[
\Vert f\Vert_{B_{\sigma}}:= \inf_{\rho \in \mathcal{G}_{\sigma,f}} \Vert f\Vert_{B_{\sigma,\rho}},
\]
in which $\mathcal{G}_{\sigma, f} = \left\{ \rho |
 f(x)=\int_{\mathbb{R}^{d+1}} \sigma(w\cdot x+b) \,\mathrm{d}\rho, \Vert f\Vert_{B_{\sigma,\rho}} < \infty , \rho \in\mathcal{M}(\mathbb{R}^{d+1}) \right\}$. This Barron space $B_{\sigma}$ has been further investigated in \cite{embeddingbarron}, particularly in terms of its embedding properties. Moreover, \cite{embeddingbarron} extends the analysis to activation functions with higher smoothness and discusses the resulting embedding characteristics.  In a related direction for spectral Barron spaces, \cite{Meng_spectrum} recently proposed a “Barron spectrum” space based on weighted Fourier norms and proved embeddings into Sobolev and Besov spaces.

%In \cite{embeddingbarron}, Heeringa et al. carefully discussed the embedding relations between the Barron spaces. However, such definition does not provide a consistent framework for the definition of the Barron space. Instead, it separately discusses the neural networks with $\mathrm{RePU}$ activation functions and those with Lipschitz activation functions, which has also become one of the motivations for us to define the generalized Barron space.
     
%The two types of Barron spaces mentioned above are the ones we discussed later on. Besides these, there are many other spaces related to them. From a functional-analytic perspective, Bartolucci et al. showed that the set of functions realized by shallow networks forms a reproducing-kernel Banach space (RKBS) \cite{RKBS}, with a norm given by the total variation of the parameter measure. \cite{Xu_sharpbound} proposed a variation space and established sharp approximation bounds for neural network approximation. These approaches provide unified language and guarantees for deep learning, but they often remain abstract.

In this work, inspired by \cite{E_Barronspace,embeddingbarron}, we introduce a generalized class of Barron spaces for shallow neural networks that unifies the aforementioned concepts of spectral Barron spaces and (extended) Barron spaces. More precisely, we propose a generalized Barron space in which the dependence on the activation function is made explicit. Our key idea is to associate each activation function \(\sigma\) with a norm function \(\varphi\), to be defined later, and define the corresponding Barron norm via a \(\varphi\)-weighted integral over the parameter measure. Specifically, we consider the following function space:
\begin{equation*}
    B_\sigma^\varphi:=\left\{ f(x)=\int_{\mathbb{R}^{d+1}}\sigma(w\cdot x+b)\,\mathrm{d}\rho(w,b) \;:\; \int_{\mathbb{R}^{d+1}} \varphi(\|w\|_1+|b|)\,\mathrm{d}|
    \rho| <\infty, \rho \in \mathcal{M}(\mathbb{R}^{d+1}) \right\},
\end{equation*}
equipped with the norm
\begin{equation*}
    \|f\|_{B_\sigma^\varphi}:=\inf_{\rho}\int_{\mathbb{R}^{d+1}} \varphi(\|w\|_1+|b|)\,\mathrm{d}|
    \rho|,
\end{equation*}
where the infimum is taken over all representing measures \(\rho\). In fact, classical Barron spaces coincide with this general formulation upon suitable choices of \(\sigma\) and \(\varphi\). As we shall demonstrate, the function \(\varphi\) governs the growth of the parameter norm required to control the behavior of \(\sigma\). This proposed framework extends beyond existing Reproducing Kernel Banach Space (RKBS) constructions by explicitly prescribing how the norm can be tailored to the choice of activation function.

Our main contributions are as follows. First, we introduce a new class of Barron spaces \(B_{\sigma}^{\varphi}\) defined via a norm function \(\varphi\), providing a broadly applicable normed framework. Second, we establish a Sobolev-type embedding theorem: for smooth activation functions \(\sigma\) and appropriately chosen \(\varphi\), every \(f\in B_{\sigma}^{\varphi}\) possesses bounded derivatives up to a certain order and thus belongs to suitable Sobolev spaces. The \(\varphi\)-weighted norm therefore makes explicit the connection between parameter complexity and classical regularity. Third, we derive embedding relations among generalized Barron spaces. These results both generalize and make precise the embeddings studied in \cite{embeddingbarron,Wu_embeddingBarron}, and they provide concrete comparison inequalities for norms associated with different pairs \((\sigma,\varphi)\). In particular, our \(\varphi\)-weighted perspective is consistent with results for spectral and extended Barron spaces. Finally, we analyze approximation rates for functions in \(B_{\sigma}^{\varphi}\), thereby recovering and extending previously known results. We also show that the \(\varphi\)-weighted representation cost can serve as an effective regularizer and derive corresponding Tikhonov error bounds for shallow neural networks.

The remainder of this paper is organized as follows. In Section \ref{se_generalBarron}, we introduce the generalized Barron space, discuss its connection with conventional Barron-type spaces, and investigate embedding properties of generalized Barron spaces. Section \ref{se_approximateandregularization} provides comprehensive quantitative approximation results and error bounds for Tikhonov regularization coupled with a penalty term associated with the generalized Barron norm. Numerical examples in Section \ref{se_numerics} illustrate the effectiveness and accuracy of the proposed framework. In particular, we demonstrate that by selecting different norm functions, one can design the properties of shallow neural networks even when the same activation functions are used. Finally, Section \ref{se_conclusion} summarizes the present work and presents prospects for scalable research. An appendix collects additional proofs that are omitted from the main text for brevity.

\section{Generalized Barron spaces and Sobolev embedding}\label{se_generalBarron}
In this section, we provide a detailed definition of the generalized Barron spaces and investigate their embedding properties. For simplicity, we restrict our attention in this paper to functions defined on the bounded domain \(\Omega := (0,1)^d \subset \mathbb{R}^d\).

\subsection{Generalized Barron spaces}\label{subse_generalBarron}
The generalized Barron space \( B_{\sigma}^{\varphi} \) is characterized by two essential components: the activation function \( \sigma \) used in shallow neural networks, and a norm function \( \varphi \) that induces a weighted norm capturing the regularity properties of the space. Throughout this paper, an activation function is assumed, at a minimum, to be a Borel-measurable map \(\sigma:\mathbb{R}\to\mathbb{C}\); in the real-valued setting, its range is restricted to \(\mathbb{R}\). This assumption ensures that \((w,b)\mapsto\sigma(w\cdot x+b)\) is measurable for every fixed \(x\in\Omega\), so the parameter integral defining a neural network is meaningful. Whenever continuity, differentiability, or Sobolev regularity of \(\sigma\) is required, the corresponding stronger assumption is stated explicitly. We now provide the precise definition of \(\varphi\).
\begin{definition}[Norm function]\label{def_indexfunction}
Given an activation function \(\sigma\) and a function \(\varphi:[0,\infty)\to(0,\infty)\), consider the following conditions:  
    \begin{enumerate}
        \item $\varphi$ is positive and non-decreasing. 
        \item There exists a constant \(C > 0\) such that  
   \[
   |\sigma(x)| \le C \varphi(|x|) \quad \text{for all } x \in \mathbb{R}. 
   \]
    \end{enumerate}
We then call \(\varphi\) a norm function for the activation function \(\sigma\).
\end{definition}
\begin{remark}
We observe that the above definition implies that the growth of the activation function \(\sigma\) is controlled by that of the norm function \(\varphi\). A detailed discussion of the choice of norm function, together with concrete examples, is provided below.
\end{remark}

Based on the definition of norm functions, we introduce the generalized Barron space as follows.
\begin{definition}[Generalized Barron space]    
Let \(\mathcal{M}(\mathbb{R}^{d+1};\mathbb{C})\) denote the space of finite complex Radon measures, equipped with the total variation norm; in the real-valued setting, complex measures may be replaced by finite signed Radon measures. Functions on \(\Omega\) are identified when they agree almost everywhere. Given an activation function \(\sigma\) and a norm function \(\varphi\), if a measurable function \(f:\Omega\to\mathbb{C}\) admits a representation of the form
\begin{equation}\label{eq_def_f}
f(x)=\int_{\mathbb{R}^{d+1}}{\sigma}(w\cdot x+b)\mathrm{d}\rho(w,b),
\qquad \text{for a.e. }x\in\Omega,
\end{equation}
    with
\begin{equation*}
    \Vert f\Vert_{B_{\sigma,\rho}^\varphi}:=\int_{\mathbb{R}^{d+1}} \varphi(\Vert w\Vert_1+|b|)\mathrm{d}|\rho|<\infty,
\end{equation*}
for some measure \(\rho\) belonging to
$$\mathcal{M}^{\varphi}(\mathbb{R}^{d+1}):=\left\{\rho\in\mathcal{M}(\mathbb{R}^{d+1};\mathbb{C}):\|\rho\|_{\mathcal M^\varphi}:=\int_{\mathbb{R}^{d+1}} \varphi(\|w\|_1+|b|) \mathrm{d}|\rho|<\infty\right\},$$
then all such functions constitute the generalized Barron space \(B_{\sigma}^{\varphi}(\Omega)\). When only real-valued functions and activations are considered, the representing measures are taken to be finite signed Radon measures.  
The corresponding generalized Barron norm on \(B_{\sigma}^{\varphi}(\Omega)\) is defined as  
\begin{equation}\label{eq_defbarron}
    \Vert f\Vert_{B_{\sigma}^\varphi(\Omega)} := \inf_{\rho \in N_{\sigma}^{-1}(f)}\int_{\mathbb{R}^{d+1}} \varphi(\Vert w\Vert_1+|b|)\mathrm{d}|\rho|,
\end{equation}
where \(N_{\sigma}\) denotes the mapping  
\begin{align*}
N_{\sigma}(\rho) := \int_{\mathbb{R}^{d+1}}\sigma(w\cdot x+b)\mathrm{d}\rho.
\end{align*}
When the domain is fixed, we abbreviate \(B_{\sigma}^{\varphi}(\Omega)\) and its norm as \(B_{\sigma}^{\varphi}\) and \(\|\cdot\|_{B_{\sigma}^{\varphi}}\), respectively.
\end{definition}

With the above definition of a generalized Barron space, we establish below that $\Vert \cdot\Vert_{B_{\sigma}^{\varphi}}$ in (\ref{eq_defbarron}) is indeed a norm and that $B_{\sigma}^{\varphi}$ is continuously embedded into $L^{\infty}$.
\begin{proposition}
Given an activation function \(\sigma\) and an associated norm function \(\varphi\), the space \(B_{\sigma}^{\varphi}(\Omega)\), equipped with the norm \(\|\cdot\|_{B_{\sigma}^{\varphi}(\Omega)}\), is a Banach space. Moreover, the continuous embedding \(B_{\sigma}^{\varphi}(\Omega) \hookrightarrow L^{\infty}(\Omega)\) holds; i.e., there exists a constant \(C>0\) such that
\begin{equation}\label{embed_Linfty}
\|f\|_{L^{\infty}(\Omega)} \leq C \|f\|_{B_{\sigma}^{\varphi}(\Omega)} \quad \text{for all } f \in B_{\sigma}^{\varphi}(\Omega).
\end{equation}
\end{proposition}
    \begin{proof}
We first show that \(\|\cdot\|_{B_{\sigma}^{\varphi}}\) defines a norm on \(B_{\sigma}^{\varphi}(\Omega)\).  
Let \(f_1, f_2 \in B_{\sigma}^{\varphi}(\Omega)\) take the form of (\ref{eq_def_f}). By definition of the norm, for any \(\varepsilon > 0\) there exist measures \(\rho_i \in \mathcal{M}^{\varphi}(\mathbb{R}^{d+1})\) (\(i = 1,2\)) such that  
\[
f_i(x) = \int_{\mathbb{R}^{d+1}} \sigma(w \cdot x + b) \, d\rho_i \quad \text{and} \quad 
\|f_i\|_{B_{\sigma,\rho_i}^{\varphi}} \le \|f_i\|_{B_{\sigma}^{\varphi}} + \varepsilon,
\]  
where \(\|f_i\|_{B_{\sigma,\rho_i}^{\varphi}} := \int_{\mathbb{R}^{d+1}} \varphi(\|w\|_1 + |b|) \, d|\rho_i|\).  
Then  
\[
f_1 + f_2 = \int_{\mathbb{R}^{d+1}} \sigma(w \cdot x + b) \, d(\rho_1 + \rho_2),
\]  
and consequently  
\[
\begin{aligned}
\|f_1 + f_2\|_{B_{\sigma}^{\varphi}} 
&\le \|f_1 + f_2\|_{B_{\sigma,\rho_1+\rho_2}^{\varphi}} \\
&= \int_{\mathbb{R}^{d+1}} \varphi(\|w\|_1 + |b|) \, d|\rho_1 + \rho_2| \\
&\le \int_{\mathbb{R}^{d+1}} \varphi(\|w\|_1 + |b|) \, d(|\rho_1| + |\rho_2|) \\
&= \|f_1\|_{B_{\sigma,\rho_1}^{\varphi}} + \|f_2\|_{B_{\sigma,\rho_2}^{\varphi}} \\
&\le \|f_1\|_{B_{\sigma}^{\varphi}} + \|f_2\|_{B_{\sigma}^{\varphi}} + 2\varepsilon.
\end{aligned}
\]  
Since \(\varepsilon > 0\) is arbitrary, the triangle inequality follows. Absolute homogeneity is straightforward to verify.  

It remains to prove the continuous embedding \(\eqref{embed_Linfty}\); this will also establish positive definiteness.  
For any \(f \in B_{\sigma}^{\varphi}(\Omega)\) and any representation  
\[
f(x) = \int_{\mathbb{R}^{d+1}} \sigma(w \cdot x + b) \, d\rho
\]  
with \(\rho \in \mathcal{M}^{\varphi}(\mathbb{R}^{d+1})\), we have, for every \(x \in \Omega\) for which the representation holds,
\[
\begin{aligned}
|f(x)| 
&\le \int_{\mathbb{R}^{d+1}} |\sigma(w \cdot x + b)| \, d|\rho| \\
&\le C \int_{\mathbb{R}^{d+1}} \varphi(|w \cdot x + b|) \, d|\rho| \quad \text{(by the growth condition } |\sigma(t)| \le C\varphi(|t|)\text{)}.
\end{aligned}
\]  
Because \(\varphi\) is non‑decreasing and \(|w \cdot x + b| \le \|w\|_1 \|x\|_\infty + |b|\), and since \(\|x\|_\infty \le 1\) for \(x \in \Omega\), we obtain  
\[
\varphi(|w \cdot x + b|) \le \varphi(\|w\|_1 + |b|).
\]  
Thus  
\[
|f(x)| \le C \int_{\mathbb{R}^{d+1}} \varphi(\|w\|_1 + |b|) \, d|\rho| = C \|f\|_{B_{\sigma,\rho}^{\varphi}}.
\]  
The right‑hand side is independent of \(x\), so  
\[
\|f\|_{L^\infty(\Omega)} \le C \|f\|_{B_{\sigma,\rho}^{\varphi}}.
\]  
Taking the infimum over all representations \(\rho\) of \(f\) yields  
\[
\|f\|_{L^\infty(\Omega)} \le C \|f\|_{B_{\sigma}^{\varphi}},
\]  
which is exactly \(\eqref{embed_Linfty}\). In particular, if \(\|f\|_{B_{\sigma}^{\varphi}} = 0\) then \(f = 0\) almost everywhere, establishing positive definiteness.  
Hence \(\|\cdot\|_{B_{\sigma}^{\varphi}}\) is a norm on \(B_{\sigma}^{\varphi}(\Omega)\).    

We now prove the completeness of \(B_{\sigma}^{\varphi}\). By \eqref{embed_Linfty}, the linear map  
\[
N_{\sigma} : \mathcal{M}^{\varphi}(\mathbb{R}^{d+1}) \to L^{\infty}(\Omega),\quad 
\rho \mapsto \int_{\mathbb{R}^{d+1}} \sigma(w\cdot x+b) \,d\rho(w,b)
\]  
is continuous; hence its kernel \(\ker(N_{\sigma})\) is a closed subspace. By construction, \(B_{\sigma}^{\varphi}\) is isometrically isomorphic to the quotient space  
\[
\mathcal{M}^{\varphi}(\mathbb{R}^{d+1}) / \ker(N_{\sigma}).
\]  
The space \(\mathcal{M}^{\varphi}(\mathbb{R}^{d+1})\) is Banach under \(\|\cdot\|_{\mathcal M^\varphi}\). Indeed, since \(\varphi\ge\varphi(0)>0\), the map
\[
\rho\longmapsto \varphi(\|w\|_1+|b|)\rho
\]
is an isometric isomorphism from \(\mathcal M^\varphi\) onto the Banach space of finite complex Radon measures with the total variation norm. Since \(\ker(N_{\sigma})\) is closed, the quotient is also a Banach space. Consequently, \(B_{\sigma}^{\varphi}\) is complete.
\end{proof}

If the activation function $\sigma$ possesses a higher degree of smoothness, then the generalized Barron space $B_{\sigma}^{\varphi}$ can be shown to have the Sobolev embedding property. This property plays a crucial role in the subsequent discussion on the applications of these spaces.
    
\begin{theorem}[Sobolev embedding of generalized Barron spaces]\label{sobolevembed}
Let \(\sigma \in W^{k,1}_{\mathrm{loc}}\) be an activation function and \(\varphi\) a norm function. Define \(\varphi_{m}(x) := \varphi(x) / \bigl(1 + \mathrm{ReLU}(x) \bigr)^{m}\), in which $\mathrm{ReLU}(x) := \max(0,x)$. Assume that  
\begin{equation}\label{cond1}
|\partial^{m}_{x} \sigma(x)| \le C \varphi_{m}(|x|), \quad  m\le k,
\end{equation}
and that \(\varphi_k\) is non‑decreasing and positive. Then for every \(f \in B_{\sigma}^{\varphi}(\Omega)\) of the form  
\[
f(x) = \int_{\mathbb{R}^{d+1}} \sigma(w \cdot x + b) \, \mathrm{d}\rho(w, b),
\]  
it holds for any multi-index \(|\alpha| \leq k\) that  
\begin{equation}\label{de_rep}
\partial_{x}^{\alpha} f = \int_{\mathbb{R}^{d+1}} w^{\alpha} \partial_x^{|\alpha|} \sigma(w\cdot x+b) \, \mathrm{d}\rho.
\end{equation}
Moreover, there holds the embedding \(B_{\sigma}^{\varphi}(\Omega) \hookrightarrow W^{k,\infty}(\Omega)\); i.e.,  
\begin{align*}
\|\cdot\|_{W^{k,\infty}(\Omega)} \lesssim \|\cdot\|_{B_{\sigma}^{\varphi}(\Omega)}.
\end{align*}
\end{theorem}
    \begin{proof}
        Assume that $f\in B_{\sigma}^{\varphi}$ and $f$ has the following form
        \begin{align*}
            f(x)=\int_{\mathbb{R}^{d+1}} \sigma(w\cdot x+b)\mathrm{d}\rho.
        \end{align*}
        By definition of weak derivative, we need to prove that 
        \begin{align*}
            \int_{\Omega} (-1)^{\vert \alpha \vert} f \partial_x^{\alpha}\psi \mathrm{d}x = \int_{\Omega}\left[\int_{\mathbb{R}^{d+1}} w^{\alpha} \partial_x^{|\alpha|}\sigma(w\cdot x+b)\mathrm{d}\rho\right] \psi\mathrm{d}x < \infty, \quad \forall \psi\in C_c^{\infty}(\Omega).
        \end{align*}
        Because $\eqref{cond1}$ holds, we derive that
        \begin{align*}
            \left|w^{\alpha} \partial_x^{|\alpha|}\sigma(w\cdot x+b) \psi(x)\right|&\le C(1+\Vert w\Vert_1)^{|\alpha|}\varphi_{|\alpha|}(|w\cdot x+b|)|\psi(x)|\\
            &\le C\varphi(\|w\|_1+|b|)|\psi(x)|,
        \end{align*}
         and the last inequality follows by the non-decreasing property of $\varphi_k$. Therefore,
        \begin{align*}
            \int_{\Omega}\int_{\mathbb{R}^{d+1}} \left|w^{\alpha} \partial_x^{|\alpha|}\sigma(w\cdot x+b)\psi\right|\mathrm{d}|\rho| \mathrm{d}x \le \int_{\Omega}\left[\int_{\mathbb{R}^{d+1}} C\varphi(\Vert w\Vert_1+|b|)\mathrm{d}|\rho| \right] |\psi|\mathrm{d}x < \infty,
        \end{align*}
        for any $\psi \in C_c^{\infty}(\Omega)$.
        By Fubini's theorem, we can derive
        \begin{align*}
            &\int_{\Omega}\left[\int_{\mathbb{R}^{d+1}} w^{\alpha} \partial_x^{|\alpha|}\sigma(w\cdot x+b)\mathrm{d}\rho\right] \psi\mathrm{d}x\\
            & \quad = \int_{\mathbb{R}^{d+1}}\left[\int_{\Omega} w^{\alpha} \partial_x^{|\alpha|}\sigma(w\cdot x+b)\psi\mathrm{d}x\right]\mathrm{d}\rho \\
            & \quad = \int_{\mathbb{R}^{d+1}}\left[\int_{\Omega} (-1)^{\vert \alpha \vert} \sigma(w\cdot x+b) \partial_x^{\alpha}\psi\mathrm{d}x\right]\mathrm{d}\rho\\
            & \quad = \int_{\Omega}\left[\int_{\mathbb{R}^{d+1}}  (-1)^{\vert \alpha \vert} \sigma(w\cdot x+b)  \mathrm{d}\rho\right]\partial_x^{\alpha}\psi\mathrm{d}x\\
            & \quad = \int_{\Omega} (-1)^{|\alpha|} f \partial_x^{\alpha}\psi \mathrm{d}x,
        \end{align*}
        which yields $\eqref{de_rep}$. Consequently we can obtain
        \begin{align*}
            |\partial_x^\alpha f(x)|&\le\int_{\mathbb{R}^{d+1}}|w^{\alpha} \partial_x^\alpha \sigma(w\cdot x+b)|\mathrm{d} |\rho|\\
            &\le C\int_{\mathbb{R}^{d+1}} (1+\|w\|_1)^{|\alpha|}\varphi(|w\cdot x+b|)/(1+|w\cdot x+b|)^{|\alpha|} \mathrm{d}|\rho|\\
            &\le C\int_{\mathbb{R}^{d+1}}\varphi(\|w\|_1+|b|)\mathrm{d}|\rho| =C\|f\|_{B_{\sigma,\rho}}^{\varphi},
        \end{align*}
        which completes the proof.
    \end{proof}

The above theorem implies that a faster-growing norm function \(\varphi\) leads to higher-order smoothness for the Barron space \(B_{\sigma}^{\varphi}\). However, Theorem 1 in \cite{embeddingbarron} shows that if the activation function satisfies \(\sigma \in C^{k}(\mathbb{R})\) and \(\partial^{k+1}\sigma \in L^{1}(\mathbb{R})\), then choosing \(\varphi(x)=1+\mathrm{ReLU}(x)\) yields the embedding \(B_{\sigma}^{1+\mathrm{ReLU}(x)} \hookrightarrow B^{k}\). Combined with the known embedding \(B^{k} \hookrightarrow W^{k,\infty}\) from \cite{Li_derivativeapprox}, it follows immediately that \(B_{\sigma}^{1+\mathrm{ReLU}(x)} \hookrightarrow W^{k,\infty}\). This observation naturally raises the question of whether the condition imposed in the above theorem is actually necessary, or whether one must consider \(B_{\sigma}^{1+\mathrm{ReLU}^{k}}\) to guarantee higher smoothness. The following remark demonstrates that the condition in the theorem is indeed necessary, and that the space \(B_{\sigma}^{1+\mathrm{ReLU}^{k}}\) plays a meaningful role in characterizing the smoothness of Barron spaces.

\begin{remark}
Recent research (e.g., \cite{Siren}) has demonstrated that oscillatory activation functions, such as the sine function, can be particularly well-suited for certain problems. Accordingly, we consider the corresponding generalized Barron space \(B_{\sin(x)}^{1+\mathrm{ReLU}(x)}\). For the activation function \(\sigma(x)=\sin(x)\), define  
\begin{align*}  
    f(x)=\sum_{i=1}^\infty 26^{-i}\pi^{-1}\sin(13^i\pi x).  
\end{align*}  
Then,  
\[  
\|f\|_{B_{\sin(x)}^{1+\mathrm{ReLU}(x)}} \le \sum_{i=1}^\infty 26^{-i}\pi^{-1}(1+13^i\pi) < \infty.  
\]  
However, differentiating term by term yields  
\begin{align*}  
    f'(x)=\sum_{i=1}^\infty 2^{-i}\cos(13^i\pi x),  
\end{align*}  
which is a classic example of a Weierstrass-type function. Consequently, we have $B_{\sin(x)}^{1+\mathrm{ReLU}(x)} \hookrightarrow W^{1,\infty}$, but $B_{\sin(x)}^{1+\mathrm{ReLU}(x)} \not\hookrightarrow W^{2,\infty}$, because the Weierstrass function lacks a weak derivative. This example illustrates the necessity of Theorem \ref{sobolevembed} and motivates the definition of $B_{\sigma}^\varphi$.

\end{remark}

\subsection{Relation with conventional Barron-type spaces}\label{subse_Barrontype}
Having defined the generalized Barron space, we now investigate its relationship with conventional Barron-type spaces. As demonstrated in \cite{embeddingbarron}, for neural networks with Lipschitz activation functions, the choice \(\varphi(x)=1+\mathrm{ReLU}(x)\) agrees with our definition. For networks using a \(\mathrm{ReLU}^s(x)\) activation, we may take \(\varphi(x)=1+\mathrm{ReLU}^t(x)\) with \(t\ge s\). This function satisfies Definition~\ref{def_indexfunction}. We show below that \(\Vert \cdot \Vert_{B^s} \simeq \Vert \cdot \Vert_{B_{\mathrm{ReLU}^s(x)}^{1+\mathrm{ReLU}^t(x)}}\); thus the generalized and extended Barron spaces coincide with equivalent norms.

\begin{proposition}\label{generalized_Barron}
With the choice of $\varphi=1+\mathrm{ReLU}^t(x)$ for $t\ge s$, there holds the identification $B_{\mathrm{ReLU}^s(x)}^{1+\mathrm{ReLU}^t(x)}=B^{s}$, and the following norm equivalence holds:
\begin{align*}
    \frac{1}{2} \| \cdot \|_{B_{\mathrm{ReLU}^s(x)}^{1+\mathrm{ReLU}^t(x)}} \le \| \cdot \|_{B^s} \le \| \cdot \|_{B_{\mathrm{ReLU}^s(x)}^{1+\mathrm{ReLU}^t(x)}}.
\end{align*}    
    \end{proposition}
    \begin{proof}
By definition, for any $f \in B^s$, we have the representation
\begin{align*}
    f(x) = \int_{\mathbb{R}^{d+1}} \mathrm{ReLU}^s(w \cdot x + b) \, d\rho,
\end{align*}
where $\rho$ is a signed Radon measure. The corresponding norms are then given by
\begin{align*}
    \| f \|_{B^s_\rho} = \int_{\mathbb{R}^{d+1}} \mathrm{ReLU}^s(\|w\|_1 + |b|) \, d|\rho|,
\end{align*}
and
\begin{align*}
    \| f \|_{B^{\varphi}_{\mathrm{ReLU}^s(x),\rho}} = \int_{\mathbb{R}^{d+1}} \big(1 + \mathrm{ReLU}^t(\|w\|_1 + |b|)\big) \, d|\rho|.
\end{align*}
Since \(u^s\le 1+u^t\) for \(u\ge0\) and \(t\ge s\), we immediately obtain \(\| f \|_{B^{\varphi}_{\mathrm{ReLU}^s(x),\rho}} \ge \| f \|_{B^s_\rho}\).

To establish the reverse inequality $\frac{1}{2}\|\cdot\|_{B_{\mathrm{ReLU}^s(x)}^\varphi} \le \|\cdot\|_{B^s}$, we observe that $f$ can be reparameterized as
\begin{align*}
    f &= \int_{\mathbb{R}^{d+1}} g^{s}\,\mathrm{ReLU}^s(g^{-1}w \cdot x + g^{-1}b) \, d\rho \\
    &= \int_{\mathbb{R}^{d+1}} \mathrm{ReLU}^s(w \cdot x + b) \, d\rho_g,
\end{align*}
where $g$ is a measurable function and $\rho_g$ is a finite signed Radon measure defined by
\begin{align*}
    \rho_g(D) = \int_{\mathbb{R^{d+1}}} g^{s} \mathbf{1}_{D}(g^{-1}w, g^{-1}b) \, d\rho,
\end{align*}
with $\rho_g(\mathbb{R}^{d+1}) < \infty$. Consequently,
\begin{align*}
    \| f \|_{B_{\mathrm{ReLU}^s(x),\rho_g}^\varphi} &= \int_{\mathbb{R}^{d+1}} \big( g^{s} + g^{s}\,\mathrm{ReLU}^t(\|g^{-1}w\|_1 + |g^{-1}b|) \big) \, d|\rho| \\
    &= \int_{\mathbb{R}^{d+1}} \big( g^{s} + g^{s-t}\,\mathrm{ReLU}^t(\|w\|_1 + |b|) \big) \, d|\rho|.
\end{align*}
Choosing \(g = \|w\|_1 + |b|\) gives the desired inequality \(\frac{1}{2}\|\cdot\|_{B_{\mathrm{ReLU}^s(x)}^{\varphi}} \le \|\cdot\|_{B^s}\).    

    \end{proof}
    
The above proposition suggests that the activation function \(\sigma\) determines the resulting generalized Barron space when its smoothness is less than that encoded by the norm function \(\varphi\). In particular, taking \(\varphi(x)=\exp(x)\) similarly gives the norm equivalence
\[
\Vert \cdot\Vert_{B^s} \simeq \Vert \cdot\Vert_{B_{\mathrm{ReLU}^s(x)}^{\exp(x)}}, \quad \forall s\in\mathbb{R}^{+}.
\]
This naturally raises the question: what happens when the activation function has higher smoothness than the norm function? We address this case in the following proposition, which establishes a connection between the general Barron space and the spectral Barron space.
\begin{proposition}\label{generalized_spectraBarron}        
Let $\varphi=(1+\mathrm{ReLU}(x))^s$ and $\sigma=\exp(i2\pi x)$. Then for $s\ge 1$, there holds $B_{\exp(i2\pi x)}^{(1+\mathrm{ReLU}(x))^s}=\mathscr{B}^s$, and the corresponding norms are equivalent:
\begin{equation*}
\|\cdot\|_{B_{\exp(i2\pi x)}^{(1+\mathrm{ReLU}(x))^s}}\simeq\|\cdot\|_{\mathscr{B}^s}.
\end{equation*}        
\end{proposition}
    \begin{proof}
For \(f\in\mathscr{B}^s\), choose an extension \(f_e\) occurring in the definition of \(\mathscr{B}^s\) and define the complex Radon measure
\[
\mathrm{d}\rho(w,b):=\widehat{f_e}(w)\,\mathrm{d}w\,\delta_0(\mathrm{d}b).
\]
Fourier inversion then gives
\[
f(x)=\int_{\mathbb{R}^{d+1}}e^{i2\pi(w\cdot x+b)}\,\mathrm{d}\rho(w,b),
\qquad x\in\Omega,
\]
and
\[
\|f\|_{B_{\exp(i2\pi x),\rho}^{(1+\mathrm{ReLU}(x))^s}}
=\int_{\mathbb{R}^d}(1+\|w\|_1)^s|\widehat{f_e}(w)|\,\mathrm{d}w.
\]
Taking the infimum over all admissible extensions yields
\[
\|f\|_{B_{\exp(i2\pi x)}^{(1+\mathrm{ReLU}(x))^s}}
\le \|f\|_{\mathscr{B}^s}.
\]

Conversely, let \(f\in B_{\exp(i2\pi x)}^{(1+\mathrm{ReLU}(x))^s}\) have a representation by \(\rho\in\mathcal{M}^{\varphi}(\mathbb{R}^{d+1})\). Let \(\pi(w,b)=w\) and define the complex measure
\[
\mu:=\pi_{\#}\bigl(e^{i2\pi b}\rho\bigr)
\quad\text{on }\mathbb{R}^d.
\]
Then \(f(x)=\int_{\mathbb{R}^d}e^{i2\pi w\cdot x}\,\mathrm{d}\mu(w)\) on \(\Omega\), and the total-variation inequality for pushforward measures gives
\[
\int_{\mathbb{R}^d}(1+\|w\|_1)^s\,\mathrm{d}|\mu|(w)
\le\int_{\mathbb{R}^{d+1}}(1+\|w\|_1+|b|)^s\,\mathrm{d}|\rho|(w,b).
\]
Choose \(\chi\in C_c^\infty(\mathbb{R}^d)\) with \(\chi=1\) on \(\Omega\), and define
\[
f_e(x):=\chi(x)\int_{\mathbb{R}^d}e^{i2\pi w\cdot x}\,\mathrm{d}\mu(w).
\]
Then \(f_e|_\Omega=f\) and \(\widehat{f_e}=\widehat{\chi}*\mu\). Since \(\widehat{\chi}\) is a Schwartz function, Fubini's theorem and
\[
(1+\|w\|_1)^s\le(1+\|w-t\|_1)^s(1+\|t\|_1)^s
\]
imply
\[
\begin{aligned}
\int_{\mathbb{R}^d}(1+\|w\|_1)^s|\widehat{f_e}(w)|\,\mathrm{d}w
&\le C_{\chi,s}\int_{\mathbb{R}^d}(1+\|t\|_1)^s\,\mathrm{d}|\mu|(t)\\
&\le C_{\chi,s}\int_{\mathbb{R}^{d+1}}(1+\|w\|_1+|b|)^s\,\mathrm{d}|\rho|(w,b).
\end{aligned}
\]
Taking the infimum over all representing measures \(\rho\) proves the reverse norm inequality. For real-valued \(f\), the representing complex measure may be chosen Hermitian symmetric.
    \end{proof}
    
\begin{remark}
As established in Proposition \ref{generalized_spectraBarron} above, the spaces \(\mathscr{B}^s\) and \(B^s\) are equivalent to generalized Barron spaces with weight function \(\varphi(x) = 1 + \operatorname{ReLU}^s(x)\). Moreover, \(\varphi\) can be replaced by \(\psi(x) = (1+\operatorname{ReLU}(x))^s\), since \(\varphi \simeq \psi\) defines the same generalized Barron space.

Now set \(\psi_s(x) := \frac{\psi(x)}{(1+\operatorname{ReLU}(x))^s} \equiv 1\), which is positive and non‑decreasing. In view of the bounds
\begin{align*}
& |\partial_x^m \exp(i 2\pi x)| \le (2\pi)^m \, (1+\mathrm{ReLU}(x))^{s-m}, \\
& \partial_x^m \operatorname{ReLU}^s(x) \le \frac{s!}{(s-m)!} \, (1+\mathrm{ReLU}(x))^{s-m},
\end{align*}
the assumptions of Theorem \ref{sobolevembed} are satisfied. Consequently, we obtain the continuous embeddings
\[
\mathscr{B}^s, \; B^s \hookrightarrow W^{s,\infty}.
\]
\end{remark}
    
\subsection{Embedding property of generalized Barron spaces}\label{subse_Barronembedding}
In the two preceding subsections, we introduced generalized Barron spaces and examined various properties of such spaces, either with a fixed activation function and varying norm functions, i.e., pairs of the form \(B_{\sigma}^{\varphi_1}\) and \(B_{\sigma}^{\varphi_2}\), or with a fixed norm function and varying activation functions, i.e., \(B_{\sigma_1}^{\varphi}\) and \(B_{\sigma_2}^{\varphi}\). In this section, motivated by results in \cite{embeddingbarron}, we extend the discussion to a more general setting and investigate embedding properties between \(B_{\sigma_1}^{\varphi_1}\) and \(B_{\sigma_2}^{\varphi_2}\), where both the activation functions and the norm functions may differ. To this end, we need the following definition in \cite{measuretheory}.
\begin{definition}[Finite signed complex kernel]\cite{measuretheory}
Let \((X, \mathscr{A})\) and \((Y, \mathscr{B})\) be measurable spaces. A map \(K:X\times\mathscr{B}\to\mathbb{C}\) is called a finite signed complex kernel from \((X, \mathscr{A})\) to \((Y, \mathscr{B})\) if
\begin{itemize}
    \item[(i)] for each fixed \(x \in X\), the map \(B \mapsto K(x, B)\) is a finite signed complex measure on \((Y, \mathscr{B})\), and  
    \item[(ii)] for each fixed \(B \in \mathscr{B}\), the map \(x \mapsto K(x, B)\) is \(\mathscr{A}\)-measurable.
\end{itemize}
We write \(|K|(x,\cdot)\) for the total variation of \(K(x,\cdot)\) and assume that \(x\mapsto |K|(x,B)\) is measurable for every \(B\in\mathscr{B}\).
\end{definition}

 In particular, the following lemma, which is analogous to \cite[Ex.~7, Sec.~2.6]{measuretheory}, plays a central role in the embedding argument below; its proof is given in Appendix~\ref{appendix_B}.  
\begin{lemma}\label{kernel_prop}
Suppose that \(K\) is a finite signed complex kernel from \((X,\mathscr{A})\) to \((Y,\mathscr{B})\), that \(\mu\) is a finite signed complex measure on \((X,\mathscr{A})\), and that
\[
\int_X |K|(x,Y)\,\mathrm{d}|\mu|(x)<\infty.
\]
Then the set function \(\nu(B):=\int_X K(x,B)\,\mu(\mathrm{d}x)\) is a finite signed complex measure and satisfies
\[
|\nu|(B)\le \int_X |K|(x,B)\,\mathrm{d}|\mu|(x).
\]
Moreover, if \(f:Y\to\mathbb{C}\) is measurable and
\[
\int_X\int_Y |f(y)|\,\mathrm{d}|K|(x,y)\,\mathrm{d}|\mu|(x)<\infty,
\]
then \(x\mapsto\int_Y f(y)\,K(x,\mathrm{d}y)\) is measurable and integrable, and
\[
\int_Y f(y)\,\nu(\mathrm{d}y)
=\int_X\left(\int_Y f(y)\,K(x,\mathrm{d}y)\right)\mu(\mathrm{d}x).
\]
\end{lemma}

Based on the above definition and lemma, we now present our main result regarding the embedding relationships between generalized Barron spaces.
\begin{theorem}\label{embed_Barron}
Let \(\sigma_1\) and \(\sigma_2\) be two activation functions, and let \(\varphi_1, \varphi_2 \) be associated norm functions. Suppose \(f \in B_{\sigma_2}^{\varphi_2}\) admits a representation  
\begin{equation}\label{int_rep_cond}
    f(x) = \int_{\mathbb{R}^{d+1}} \sigma_2(w \cdot x + b) \, d\rho(w,b), \qquad x \in \Omega,
\end{equation}
where \(\rho\in\mathcal{M}^{\varphi_2}(\mathbb{R}^{d+1})\). Assume that \((w,b)\mapsto\mu_{(w,b)}\) is a finite signed complex kernel from \(\mathbb{R}^{d+1}\) to itself such that for every \((w,b) \in \operatorname{supp} \rho\) and \(x \in \Omega\),  
\[
\sigma_2(w \cdot x + b) = \int_{\mathbb{R}^{d+1}} \sigma_1(\omega \cdot x + \beta) \, d\mu_{(w,b)}(\omega,\beta),
\]  
and the following estimate holds uniformly for \((w,b) \in \operatorname{supp} \rho\):  
\begin{equation}\label{meas_bdd_cond}
    \int_{\mathbb{R}^{d+1}} \varphi_1\bigl(\|\omega\|_1 + |\beta|\bigr) \, d|\mu_{(w,b)}|(\omega,\beta) \le M \, \varphi_2\bigl(\|w\|_1 + |b|\bigr),
\end{equation}
with a constant \(M > 0\) independent of \((w,b)\). Then there exists a finite signed complex Radon measure \(\nu\) on \(\mathbb{R}^{d+1}\) such that  
\[
f(x) = \int_{\mathbb{R}^{d+1}} \sigma_1(\omega \cdot x + \beta) \, d\nu(\omega,\beta), \qquad x \in \Omega,
\]  
and  
\[
\|f\|_{B_{\sigma_1,\nu}^{\varphi_1}} \le M \, \|f\|_{B_{\sigma_2,\rho}^{\varphi_2}}.
\]  

Consequently, for this particular \(f\) and the chosen representing measure
\(\rho\), there holds \(f\in B_{\sigma_1}^{\varphi_1}\) and
\[
\|f\|_{B_{\sigma_1}^{\varphi_1}}
\le M\|f\|_{B_{\sigma_2,\rho}^{\varphi_2}}.
\]
Suppose that there exists a sequence of representing measures
\(\{\rho_n\}_{n\in\mathbb N}\) of \(f\) such that
\[
\|f\|_{B_{\sigma_2,\rho_n}^{\varphi_2}}
\longrightarrow
\|f\|_{B_{\sigma_2}^{\varphi_2}},
\]
and that the above assumptions hold for each \(\rho_n\), with the same
constant \(M\). Then, passing to the limit as \(n\to\infty\), there holds
\[
\|f\|_{B_{\sigma_1}^{\varphi_1}}
\le M\|f\|_{B_{\sigma_2}^{\varphi_2}}.
\]
If, moreover, every \(f\in B_{\sigma_2}^{\varphi_2}\) admits such a
sequence for which the above assumptions hold with a constant \(M\)
independent of both \(f\) and \(n\), then
\[
B_{\sigma_2}^{\varphi_2}
\hookrightarrow B_{\sigma_1}^{\varphi_1}
\]
continuously, with embedding norm at most \(M\).

\end{theorem}
    \begin{proof}
        Given any $f\in B_{\sigma_2}^{\varphi_2}$, $f(x)=\int_{\mathbb{R}^{d+1}} \sigma_2(w\cdot x+b)\mathrm{d}\rho$, by Lemma \ref{kernel_prop}, we have
        \begin{align*}
            f(x)&=\int_{\mathbb{R}^{d+1}} \sigma_2(w\cdot x+b)\mathrm{d}\rho\\
                &=\int_{\mathbb{R}^{d+1}} \left[\int_{\mathbb{R}^{d+1}}\sigma_1\left(\omega\cdot x+\beta\right) \mathrm{d}\mu_{(w,b)}\right]\mathrm{d}\rho\\
                &=\int_{\mathbb{R}^{d+1}}\sigma_1(\omega\cdot x+\beta)\mathrm{d}\nu(\omega,\beta)
        \end{align*}
        where the finite signed complex measure \(\nu\) is given by \(\nu(E):=\int_{\mathbb{R}^{d+1}}\mu_{(w,b)}(E)\,\mathrm{d}\rho(w,b)\), as in Lemma~\ref{kernel_prop}.

    Therefore, we can derive
        \begin{align*}
            \|f\|_{B_{\sigma_1,\nu}^{\varphi_1}} &=\int_{\mathbb{R}^{d+1}}\varphi_1(\|w \|_1+|b|)\mathrm{d}|\nu|\\
            &\le \int_{\mathbb{R}^{d+1}} \left[\int_{\mathbb{R}^{d+1}}\varphi_1\left(\|\omega\|_1+|\beta|\right) \mathrm{d}|\mu_{(w,b)}|\right]\mathrm{d}|\rho|\\
            &\le M\int_{\mathbb{R}^{d+1}}\varphi_2\left(\|w\|_1+|b|\right) \mathrm{d}|\rho|\\
            &=M\|f\|_{B_{\sigma_2,\rho}^{\varphi_2}}
        \end{align*}
    which ends the proof. 
    \end{proof}
    
We present an extended discussion on the application of the above theorem.
\begin{remark}
There are various approaches to deriving integral representations. In \cite{embeddingbarron}, such representations are obtained via Taylor expansion; they can also be obtained via the Fourier transform for activation functions of the form \(\exp(i2\pi x)\), or via the wavelet transform for a broader class of activation functions.   
As an illustration, using the Taylor expansion of \(\exp(i w \cdot x)\), we obtain the embedding \(\mathscr{B}^{m+1} \hookrightarrow B^m\) for every \(m \in \mathbb{N}\). The proof proceeds as follows. For any fixed \(w\), the exponential function admits the Taylor expansion with integral remainder:
\begin{align}\label{exp_int_rep}
e^{i w \cdot x} 
&= \sum_{k=0}^m \frac{i^k}{k!} (w \cdot x)^k 
   + \int_0^{w \cdot x} \frac{i^{m+1} e^{i\beta}}{m!} (w \cdot x - \beta)^m \, d\beta \\
&= \sum_{k=0}^m \frac{i^k}{k!} (w \cdot x)^k 
   + \frac{i^{m+1}}{m!} \int_0^{c\|w\|_1} e^{i\beta} \, \mathrm{RePu}^m(w \cdot x - \beta) \, d\beta \nonumber\\
&\quad + (-1)^{m-1} \frac{i^{m+1}}{m!} \int_0^{c\|w\|_1} e^{-i\beta} \, \mathrm{RePu}^m(-w \cdot x - \beta) \, d\beta, \nonumber
\end{align}
where the first term is a polynomial of degree \(m\), which can be expressed as a linear combination of \(\{(w_i \cdot x + b_i)^k\}_{i=1}^{m+1}\) for suitably chosen \(\{(w_i, b_i)\}_{i=1}^{m+1}\). Therefore, $\eqref{int_rep_cond}$ is satisfied.

Now we turn to verify such representation satisfies $\eqref{meas_bdd_cond}$. The associated Barron norm of the measure corresponding to the first term of $\eqref{exp_int_rep}$ can be bounded by \((1 + \|w\|_1 + |b|)^m\) due to homogeneity in \(w\).  

Denoting the measure for the second term of $\eqref{exp_int_rep}$ as \(\mu\) and adapting the notation of Theorem $\ref{embed_Barron}$, we have  
\[
\int_{\mathbb{R}^{d+1}} (1 + \|\omega\|_1 + |\beta|)^m \, d|\mu| 
\le \frac{2}{m!} \int_0^{c\|w\|_1} (1 + \|w\|_1 + \beta)^m \, d\beta 
\le M (1 + \|w\|_1)^{m+1},
\]
where \(M\) is a constant. Applying Theorem \ref{embed_Barron} completes the proof.
    \end{remark}

A more compelling example of embedding arises in the context of spectral Barron spaces. In \cite{Xu_highorder}, an exponential variant of these spaces is introduced:  
\begin{equation}\label{eq_expspectralbarron}
    \mathscr{B}_{\beta,c}(\Omega) :=\left\{ f:\Omega \rightarrow\mathbb{R} \;\Bigg|\; \|f\|_{\mathscr{B}_{\beta,c}(\Omega)}:=\inf_{f_e|_{\Omega}=f} \int_{\mathbb{R}^d} \exp\bigl(c\|\xi\|_1^\beta\bigr)\,|\hat{f_e}(\xi)|\,\mathrm{d}\xi <\infty \right\},  
\end{equation}  
where the infimum is taken over all extensions \(f_e \in L^1(\mathbb{R}^d)\) of \(f\). Functions belonging to such Barron spaces $\mathscr{B}_{\beta,c}(\Omega)$ in (\ref{eq_expspectralbarron}) can be effectively approximated by neural networks that employ the cosine (\(\cos\)) activation function.

In our framework, we define  
\[
\mathscr{B}_{\beta,c}(\Omega) = B_{\exp(i2\pi x)}^{\exp(c x^\beta)}(\Omega).
\]  
Let  
\[
\sigma(x) = \frac{1}{1+\exp(-x)}
\]  
denote the sigmoid activation function. Using the embedding theorem stated above, we shall investigate how to choose \(\varphi\) such that  
\[
B_{\sigma}^{\varphi} \hookrightarrow \mathscr{B}_{\beta,c},
\]
which is summarized in the following proposition. 
\begin{proposition}
For any $\kappa>0$, there exists $c_1,\varepsilon>0$, such that
$$B_{\frac{1}{1+\exp(-x)}}^{\exp(c_1x^{\kappa})}(\Omega)\hookrightarrow \mathscr{B}_{\frac{\kappa}{1+\kappa},\varepsilon}(\Omega). $$
\end{proposition}

\begin{proof}
We first obtain an integral representation of \(\sigma\) in terms of \(\exp(i2\pi x)\). Let \(\chi \in C_c^\infty(\mathbb{R}^d)\) satisfy \(\chi = 1\) on \(\Omega\), and define
\[
h_{w,b}(x):=\sigma(w\cdot x+b)\chi(x),
\qquad k_{w,b}:=\widehat{h_{w,b}}.
\]
Fourier inversion gives, for \(x\in\Omega\),
\[
\sigma(w\cdot x+b)
=\int_{\mathbb{R}^d}k_{w,b}(\xi)e^{i2\pi \xi\cdot x}\,\mathrm{d}\xi.
\]
With the Fourier-transform convention used here, the one-dimensional sigmoid satisfies, in the sense of tempered distributions,
\[
\widehat{\sigma}(t)=\frac{1}{2}\delta_0(t)-\frac{i\pi}{\sinh(2\pi^2t)}.
\]
Consequently,
\[
k_{w,b}(\xi)=\frac{1}{2}\widehat{\chi}(\xi)
-i\pi\,\mathrm{P.V.}\int_{\mathbb{R}}
\frac{\widehat{\chi}(\xi-tw)e^{i2\pi bt}}{\sinh(2\pi^2t)}\,\mathrm{d}t.
\]

To find an appropriate norm function \(\varphi\), we estimate the decay of \(k_{w,b}\). This requires a cutoff \(\chi\) whose Fourier transform decays sufficiently rapidly. According to Theorem~1.6.1 in \cite{Rodino1993}, if \(f\) belongs to the Gevrey class of order \(s\), then \(\widehat f\) decays like \(\exp(-\varepsilon^{\prime}\|\xi\|_1^{1/s})\) for some \(\varepsilon^{\prime}>0\); the same estimate applies to the Fourier transforms of polynomial multiples of \(f\).
We construct a function $\chi \in C_c^\infty(\mathbb{R}^d)$ such that $\chi = 1$ on $\Omega$ and $\chi$ belongs to the Gevrey class of order $\frac{1+\kappa}{\kappa}$ as follows.

Let
\begin{align*}
    \Phi_\kappa(t)=\begin{cases}
        0& t\le 0,\\
        1& t\ge 1,\\
        \dfrac{\int_0^t\Omega_\kappa(\tau)\mathrm{d}\tau}{\int_0^1\Omega_\kappa(\tau)\mathrm{d}\tau} & t\in (0,1)
    \end{cases}
\end{align*}
with 
\begin{align*}
    \Omega_\kappa(t)=\begin{cases}
        0 & t\notin (0,1),\\
        \exp\left(\dfrac{-1}{[(1-t)t]^\kappa}\right) & t\in (0,1).
    \end{cases}
\end{align*}
It can be verified that
\[
g_\kappa(t):=\Phi_\kappa(t+1)\Phi_\kappa(2-t)
\]
belongs to the Gevrey class of order \((1+\kappa)/\kappa\), equals \(1\) on \([0,1]\), and is supported in \([-1,2]\). Therefore,
\[
\chi(x):=\prod_{j=1}^d g_\kappa(x_j)
\]
belongs to the same Gevrey class, has compact support, and equals \(1\) on \(\Omega=(0,1)^d\).

We choose 
\begin{align}\label{eq_varphi1}
\varphi_1(x) = \exp\!\Bigl(\varepsilon x^{\frac{\kappa}{1+\kappa}}\Bigr), 
\end{align}
where $\varepsilon = \varepsilon^{\prime}/2$. Then, there holds
\begin{align*}
&\int_{\mathbb{R}^{d}}\varphi_1(\|\xi\|_1)|k_{w,b}(\xi)|\mathrm{d}\xi\\
\le&\int_{\mathbb{R}^{d}}\varphi_1(\|\xi\|_1)\left[\frac{1}{2}|\hat{\chi}(\xi)|+\pi\left|\mathrm{P.V.}\int_{[-1,1]}\frac{\hat{\chi}(\xi-tw)\exp(i2\pi bt)}{\sinh(2\pi^2t)}\mathrm{d}t\right|+\pi\int_{[-1,1]^c}\frac{|\hat{\chi}(\xi-tw)|}{|\sinh(2\pi^2t)|}\mathrm{d}t\right]\mathrm{d}\xi.
\end{align*}
Denote the two integrals over \(|t|\le1\) and \(|t|>1\) by \(I_1(\xi;w,b)\) and \(I_2(\xi;w)\), respectively. We bound them separately. For the principal-value term, the oddness of \(1/\sinh(2\pi^2t)\) gives
\begin{align*}
|I_1|&=\left|\mathrm{P.V.}\int_{[-1,1]}\frac{\hat{\chi}(\xi-tw)\exp(i2\pi bt)}{\sinh(2\pi^2t)}\mathrm{d}t\right|\\
&=\left|\int_{[-1,1]}\frac{\hat{\chi}(\xi-tw)\exp(i2\pi bt)-\hat{\chi}(\xi)}{\sinh(2\pi^2t)}\mathrm{d}t\right|\\
&\le C \sup_{t\in[-1,1]}\left|-w\cdot \nabla \hat{\chi}(\xi-tw)\exp(i2\pi b t)+i2\pi b\hat{\chi}(\xi-tw)\exp(i2\pi bt) \right|.
\end{align*}
Because both \(\chi\) and \(x\chi(x)\) are compactly supported Gevrey functions of the same order, \(\widehat{\chi}\) and \(\nabla\widehat{\chi}\) satisfy the same exponential decay estimate. Therefore, uniformly for \(|t|\le1\), by set \(q:=\kappa/(1+\kappa)\), we derive 
\begin{align*}
|I_1|&\le C(\|w\|_1+|b|)\exp(-2\varepsilon \|\xi-tw\|_1^{q} )\\
& \le C(\|w\|_1+|b|)\exp(-2\varepsilon\|\xi\|_1^{q}+2\varepsilon\|w\|_1^{q}).
\end{align*}
We further turn to bound $|I_2|$,
\begin{align*}
|I_2|&\le C\int_{|t|>1}|\hat{\chi}(\xi-tw)|\exp(-2\pi^2 |t|)\mathrm{d}t\\
&\le C\int_{|t|>1} \exp(-2\varepsilon\|\xi-tw\|_1^{q}-2\pi^2 |t|)\mathrm{d}t\\
&\le C\exp(-2\varepsilon\|\xi\|_1^{q})\int_{|t|>1}\exp(2\varepsilon |t|^{q}\|w\|_1^{q}-2\pi^2 |t|)\mathrm{d}t\\
&\le C\exp(-2\varepsilon\|\xi\|_1^{q}) \exp(c_1\|w\|_1^\kappa),
\end{align*}
where the last inequality follows from Laplace's method. 

Therefore, we can derive
\begin{align*}
&\int_{\mathbb{R}^{d}}\varphi_1(\|\xi\|_1)|k_{w,b}(\xi)|\mathrm{d}\xi\\
\le &C\int_{\mathbb{R}^{d}}\varphi_1(\|\xi\|_1)\exp(-2\varepsilon\|\xi\|_1^{\frac{\kappa}{1+\kappa}})\left[1+(\|w\|_1+|b|)\exp(2\varepsilon\|w\|_1^{\frac{\kappa}{1+\kappa}})+ \exp(c_1\|w\|_1^\kappa)\right]\mathrm{d}\xi\\
\le & C\int_{\mathbb{R}^{d}}\exp(-\varepsilon\|\xi\|_1^{\frac{\kappa}{1+\kappa}})\left[1+(\|w\|_1+|b|)\exp(2\varepsilon\|w\|_1^{\frac{\kappa}{1+\kappa}})+ \exp(c_1\|w\|_1^\kappa)\right]\mathrm{d}\xi\\
\le & C\exp\left(c_1(\|w\|_1+|b|)^\kappa\right).
\end{align*}
        
Let \(\varphi_2(t)=\exp(c_1 t^\kappa)\). We have shown that  
\begin{align*}
\int_{\mathbb{R}^{d}}\varphi_1(\|\xi\|_1)|k_{w,b}(\xi)|\mathrm{d}\xi\lesssim \varphi_2(\|w\|_1+|b|),
\end{align*}
where \(\varphi_1\) is as defined in (\ref{eq_varphi1}).  
Define a complex kernel on \(\mathbb{R}^{d+1}\) by
\[
\mu_{(w,b)}(E)
:=\int_{\mathbb{R}^d}\mathbf{1}_{E}(\xi,0)k_{w,b}(\xi)\,\mathrm{d}\xi.
\]
The map \((w,b)\mapsto k_{w,b}(\xi)\) is measurable for every \(\xi\), so \((w,b)\mapsto\mu_{(w,b)}\) is a complex kernel. The Fourier representation above and the weighted estimate show that this kernel satisfies the representation and bound required in Theorem~\ref{embed_Barron}. Therefore,
\[
B_{\frac{1}{1+\exp(-x)}}^{\exp(c_1x^{\kappa})}(\Omega)
\hookrightarrow \mathscr{B}_{\frac{\kappa}{1+\kappa},\varepsilon}(\Omega).
\]
\end{proof}

\section{Approximation Rate and Regularization Method}\label{se_approximateandregularization}
In this section, we investigate the quantitative approximation properties of generalized Barron spaces and explore their applications to stable numerical reconstruction via regularization methods. By exploiting the structural properties of the underlying Banach spaces, we establish dimension‑independent approximation rates and derive corresponding error bounds for the Tikhonov regularization.

    \subsection{Approximation Rates within Type-$p$ Spaces}
    The capacity of shallow neural networks to approximate functions in $B_{\sigma}^{\varphi}$ is intrinsically linked to the property of the target space $X$. We utilize the concept of type-$p$ Banach spaces, with $p\in (1,2]$, to quantify this capacity. To this end, we recall the following definition of a type-$p$ space.
    \begin{definition}[Type-$p$ space]
A Banach space \(X\) is called a type-\(p\) space if there exists a constant \(T_p(X)<\infty\), independent of \(n\) and of the vectors, such that for every finite sequence \(\{x_i\}_{i=1}^n\subset X\) and every sequence of i.i.d. Rademacher random variables \(\{\varepsilon_i\}_{i=1}^n\) (i.e., \(\mathbb{P}(\varepsilon_i=\pm1)=\frac12\)), the following inequality holds:
    \begin{equation}\label{eq_typep}
        \mathbb{E}\left\Vert \sum_{i=1}^n \varepsilon_i x_i \right\Vert_X^p
        \le T_p(X)^p\sum_{i=1}^n\Vert x_i\Vert_X^p.
    \end{equation}
    \end{definition}

It was shown in \cite{Xu_sharpbound} that the standard Sobolev space $H^k(\Omega)$ is a type-$2$ space. By a slight modification of the proof therein, one can further show that the Sobolev space $W^{k,p}(\Omega)$ is type-$\min(2,p)$; see Proposition~\ref{sobolev_is_typep} in the Appendix. This property enables the application of the Maurey–Barron sampling argument to the integral representation of functions in $B_{\sigma}^{\varphi}$. Let $\mathcal{F}_n(\sigma)$ denote the set of neural networks with $n$ neurons and activation function $\sigma$, i.e., functions of the form
\[
f_n(x) = \sum_{i=1}^n a_i \, \sigma(w_i \cdot x + b_i).
\]
We then establish the following quantitative approximation rate for the generalized Barron space.

\begin{theorem}[Approximation rate]\label{thm_apptypep}
Assume that \(X\) is a type-\(p\) Banach space and that \(B_{\sigma}^{\varphi} \hookrightarrow X\). For any \(f \in B_{\sigma}^{\varphi}\) and any \(\varepsilon > 0\), there exists a Radon measure \(\rho\) such that
\[
f(x) = \int_{\mathbb{R}^{d+1}} \sigma(w \cdot x + b) \, \mathrm{d}\rho
\]
and
\[
\| f \|_{B_{\sigma ,\rho}^{\varphi}} \le (1+\varepsilon) \| f \|_{B_{\sigma}^{\varphi}}.
\]
Then, for every \(n \in \mathbb{N}\), there exists \(f_n \in \mathcal{F}_n(\sigma)\) satisfying
\[
\Vert f - f_n \Vert_X \le C \Vert f \Vert_{B_\sigma^{\varphi}} \, n^{\frac{1}{p} - 1},
\]
and, for some Radon measure \(\rho_n\),
\[
\| f_n \|_{B_{\sigma,\rho_n}^{\varphi}} = \| f \|_{B_{\sigma,\rho}^{\varphi}}.
\]
    \end{theorem}

The proof follows a Maurey-type sampling argument as in \cite{Xu_sharpbound}, and we briefly outline it below. Choose \(\rho\) such that \(f\) admits the integral representation above with 
$\|f\|_{B_{\sigma,\rho}^{\varphi}}\le (1+\varepsilon)\|f\|_{B_{\sigma}^{\varphi}}$.
After normalizing the representing measure, we obtain an \(X\)-valued random variable \(\hat{X}\) with \(\mathbb{E} \hat{X} = f\). Let \(\hat{X}_1,\dots,\hat{X}_n\) be i.i.d. copies of \(\hat{X}\), and define
\[
Z_n=\frac1n\sum_{i=1}^n \hat{X}_i.
\]
Using the type-$p$ property (\ref{eq_typep}) of $X$, one gets
\[
\mathbb E\|Z_n-f\|_X^p \lesssim n^{1-p}\|f\|_{B_{\sigma,\rho}^{\varphi}}^p.
\]
Hence, there exists a realization \(f_n \in \mathcal{F}_n(\sigma)\) such that
\[
\|f-f_n\|_X \le C\|f\|_{B_{\sigma}^{\varphi}}\,n^{\frac1p-1}.
\]
The equality for the corresponding Barron norm follows directly from the construction of the empirical representing measure \(\rho_n\).

Following a similar scheme to that in \cite{Xu_sharpbound}, one can obtain a higher approximation rate by using the smoothness of the activation function \(\sigma\) and the decay of the representing measure \(\rho\). We present the relevant definition and approximation result below. 

\begin{definition}\label{def_class_s}
Let $X$ be a Banach space and $U \subset \mathbb{R}^d$ be an open set. For $s = k + \alpha > 0$ with $k \in \mathbb{N}$ and $0 < \alpha \le 1$, a function $F: U \to X$ is said to be of smoothness class $s$, denoted by $F \in \mathrm{Lip}_{\infty}(s, X)$, if for every $\xi \in X^*$, the function $f_\xi: U \to \mathbb{R}$ defined by
\begin{align*}
    f_\xi(x)=\langle \xi,F(x)\rangle,
\end{align*}
satisfies
\begin{align*}
|f_\xi|_{\mathrm{Lip}_\infty(s,X)}=\dfrac{|D^kf_\xi(x)-D^kf_\xi(y)|}{|x-y|^\alpha}\le C\|\xi\|_{X^{*}},
\end{align*}
for some constant $C < \infty$ independent of $\xi$ and $x, y$ (with $x \neq y$). The smallest such constant $C$ is the seminorm $|F|_{\mathrm{Lip}_{\infty}(s, X)}$.
\end{definition}
        
\begin{theorem}\label{higher_app}
Let $s > 0$ and let $X$ be a type-$p$ Banach space such that $B_{\sigma}^{\varphi} \hookrightarrow X$. Define $P_{\sigma} : \mathbb{R}^{d+1} \to X$ by $P_{\sigma}(w,b) = \sigma(w \cdot x + b)$. For any $f \in B_{\sigma}^{\varphi}$ and any $\varepsilon > 0$, there exists a Radon signed measure $\rho$ such that $f = N_{\sigma}(\rho)$ and        
$\| f\|_{B_{\sigma ,\rho}^{\varphi}}\le(1+\varepsilon)\|f\|_{B_{\sigma}^{\varphi}}$. Denote 
\begin{align*}
\theta(R):=\int_{\mathbb{R}^{d+1}\setminus[-\frac{R}{2},\frac{R}{2}]^{d+1}}\varphi(\|w\|_1+|b|)\mathrm{d}|\rho|.
\end{align*}

If $\theta(R) \lesssim R^{-\alpha}$ and $P_{\sigma} / \varphi(\|w\|_1 + |b|) \in \mathrm{Lip}_{\infty}(s, X)$ in the sense of Definition \ref{def_class_s}, then for any $n \in \mathbb{N}$, letting $N = (K+2)n$ for a constant $K>0$ as in \cite{Xu_sharpbound}, there exists a neural network $f_N \in \mathcal{F}_N(\sigma)$ and a constant $C > 0$ such that        
\begin{align*}
    \Vert f-f_N\Vert_X\le Cn^{-\frac{\alpha}{s+\alpha}\cdot\frac{s}{d+1}+\frac{1}{p}-1}\|f\|_{B_{\sigma,\rho}^{\varphi}}
\end{align*}
and
\begin{equation*}
    \|f_N\|_{B_{\sigma,\rho_N}^{\varphi}}\le C\| f\|_{B_{\sigma ,\rho}^{\varphi}}
\end{equation*}
for some Radon measure $\rho_N$.
\end{theorem}
We note that the proof follows the line of argument in \cite{Xu_sharpbound} and can be divided into a three-stage scheme: (i) split the parameter space into "near" (compact) and "far" (tail) regions according to the $\varphi$-mass of the representing measure; (ii) approximate the near part via strategic sampling as in \cite{Xu_sharpbound} and the far part by simple Monte Carlo averaging; (iii) balance the two parts by adjusting the scale of the near region, thereby obtaining the approximation rate. For brevity, we omit the detailed proof.

A simple and useful mechanism to obtain the tail decay appearing in Theorem \ref{higher_app} is to compare two norm functions. More precisely, if $\varphi_1$ and $\varphi_2$ are two norm functions for $\sigma$ and there exists $\alpha > 0$ such that
\[
\frac{\varphi_1(x)}{\varphi_2(x)} \lesssim x^{-\alpha},
\]
then for any $f \in B_{\sigma}^{\varphi_2}$ with representing measure $\rho$ (so that $f = N_\sigma(\rho)$), we have $f \in B_{\sigma}^{\varphi_1}$, and the $\varphi_1$-mass of the complement of the cube satisfies the polynomial bound
\[
\int_{\mathbb{R}^{d+1} \setminus [-R/2, R/2]^{d+1}} \varphi_1\big(\|w\|_1 + |b|\big) \, d|\rho|(w,b) \;\lesssim\; R^{-\alpha}.
\]
In particular, taking $\varphi = \varphi_1$ and $\theta(R)$ as in Theorem \ref{higher_app} yields $\theta(R) \lesssim R^{-\alpha}$. Hence this comparison provides a sufficient condition for the tail decay hypothesis used to derive the higher approximation rate in Theorem \ref{higher_app}.

\subsection{Tikhonov Regularization and Error Bound Analysis}
In this subsection, we explore another application of the generalized Barron norm to Sobolev training, where the unknown function and its derivatives can be recovered simultaneously. This application is achieved by introducing a Tikhonov-type regularization, in which the generalized Barron norm serves as a penalty term.
Whenever $f=N_\sigma(\rho)$ for some $\rho\in\mathcal{M}^{\varphi}(\mathbb{R}^{d+1})$, we write
\[
c_{\sigma,\rho}^{\varphi}(f):=\|f\|_{B_{\sigma,\rho}^{\varphi}}
=\int_{\mathbb{R}^{d+1}}\varphi(\|w\|_1+|b|)\,\mathrm{d}|\rho|.
\]
For a neural network $g(x)=\sum_{i=1}^n a_i\,\sigma(w_i\cdot x+b_i)\in\mathcal{F}_n(\sigma)$, we denote its discrete representing measure by
\[
\rho_n:=\sum_{i=1}^n a_i\,\delta_{(w_i,b_i)},
\]
so that $c_{\sigma,\rho_n}^{\varphi}(g)=\sum_{i=1}^n |a_i|\,\varphi(\|w_i\|_1+|b_i|)$. We define the generalized Barron regularizer \(\mathcal{R}(g)\) associated with the norm \(\|\cdot\|_{B_{\sigma}^{\varphi}}\) as
\[
\mathcal{R}(g) = \left[ c_{\sigma,\rho_n}^{\varphi}(g) \right]^p
= \left[ \sum_{i=1}^n |a_i| \, \varphi(\|w_i\|_1 + |b_i|) \right]^p.
\]

Consider the problem of reconstructing \(f \in B_{\sigma}^{\varphi}\) from noisy data \(f^{\delta}\) satisfying \(\|f - f^{\delta}\|_{L^p(\Omega)} \le \delta\). We define the Tikhonov functional
\begin{equation}\label{eq_Tikfunctional}
    J_{\lambda}(g) := \|g - f^{\delta}\|_{L^p(\Omega)}^p + \lambda \mathcal{R}(g).
\end{equation}
Under the condition \(\lim_{t \to \infty} \varphi(t) = \infty\), the functional \(J_{\lambda}\) is coercive and therefore admits a minimizer \(f_{n, \lambda}^{\delta} \in \mathcal{F}_n(\sigma)\). For this minimizer $f_{n, \lambda}^{\delta}$, we derive the following error bound.
\begin{proposition}\label{prop_regularizationerror}
Let $f \in B_{\sigma}^{\varphi}$, and let $\{f_n\}_{n=1}^{\infty}\subset\mathcal{F}_n(\sigma)$ be a sequence of approximants of $f$ achieving a rate $r_n$, i.e., $\|f - f_n\| \le r_n$. Write $\rho_n:=\sum_{i=1}^n a_i^{(n)}\,\delta_{(w_i^{(n)},b_i^{(n)})}$ for the discrete representing measure of $f_n$, and assume $c_{\sigma,\rho_n}^{\varphi}(f_n)\le c_{\sigma,\rho}^{\varphi}(f)$ for some representing measure $\rho$ of $f$. Suppose $\lim_{x \to \infty} \varphi(x) = \infty$. Then any minimizer
\[
f_{n,\lambda}^{\delta} = \sum_{i=1}^n a_{i,\lambda}^{\delta} \, \sigma(w_{i,\lambda}^{\delta} \cdot x + b_{i,\lambda}^{\delta})
\]
of the Tikhonov functional (\ref{eq_Tikfunctional}) satisfies
\[
\|f - f_{n,\lambda}^{\delta}\|_{L^p(\Omega)} \le 2\delta + r_n + \lambda^{\frac{1}{p}} c_{\sigma,\rho}^{\varphi}(f)
\]
and
\[
c_{\sigma,\rho_{n,\lambda}^{\delta}}^{\varphi}(f_{n,\lambda}^{\delta}) \le \frac{\delta + r_n}{\lambda^{\frac{1}{p}}} + c_{\sigma,\rho}^{\varphi}(f),
\]
where $\rho$ is any representing measure of $f$ and $\rho_{n,\lambda}^{\delta}:=\sum_{i=1}^n a_{i,\lambda}^{\delta}\,\delta_{(w_{i,\lambda}^{\delta},b_{i,\lambda}^{\delta})}$ is the discrete representing measure of $f_{n,\lambda}^{\delta}$.
\end{proposition}
\begin{proof}
By the minimizing property of \(f_{n,\lambda}^{\delta}\), we obtain the estimate
\[
\|f^{\delta} - f_{n,\lambda}^{\delta}\|_{L^p(\Omega)}^p \le \|f^{\delta} - f_n\|_{L^p(\Omega)}^p + \lambda \bigl[c_{\sigma,\rho_n}^{\varphi}(f_n)\bigr]^p,
\]
and hence
\begin{align*}
\|f^{\delta} - f_{n,\lambda}^{\delta}\|_{L^p(\Omega)} 
&\le \|f^{\delta} - f_n\|_{L^p(\Omega)} + \lambda^{\frac{1}{p}} c_{\sigma,\rho_n}^{\varphi}(f_n) \\
&\le \delta + \|f - f_n\|_{L^p(\Omega)} + \lambda^{\frac{1}{p}} c_{\sigma,\rho_n}^{\varphi}(f_n) \\
&\le \delta + r_n + \lambda^{\frac{1}{p}} c_{\sigma,\rho}^{\varphi}(f),
\end{align*}
Consequently,
\[
\|f - f_{n,\lambda}^{\delta}\|_{L^p(\Omega)} 
\le \|f^{\delta} - f_{n,\lambda}^{\delta}\|_{L^p(\Omega)} + \|f - f^{\delta}\|_{L^p(\Omega)} 
\le 2\delta + r_n + \lambda^{\frac{1}{p}} c_{\sigma,\rho}^{\varphi}(f).
\]
Similarly, we obtain
\[
\lambda^{\frac{1}{p}} c_{\sigma,\rho_{n,\lambda}^{\delta}}^{\varphi}(f_{n,\lambda}^{\delta}) 
\le \delta + r_n + \lambda^{\frac{1}{p}} c_{\sigma,\rho}^{\varphi}(f).
\]
\end{proof}

Furthermore, if the activation function $\sigma$ and the norm function $\varphi$ satisfy the smoothness conditions in Theorem \ref{sobolevembed}, we can obtain the error bound for the derivatives.
    
\begin{theorem}\label{para_choice_thm}
If the regularization parameter $\lambda$ is chosen such that
\begin{equation}\label{para_choice}
\delta + r_n \asymp \lambda^{\frac{1}{p}},
\end{equation}
and $\sigma, \varphi$ satisfy the conditions in Theorem \ref{sobolevembed} and Proposition \ref{prop_regularizationerror}, then there exists a constant $C_1$ such that for $0 \le m \le k$, the following error bound holds:
\begin{equation}\label{eq_Tikerror}
\|f - f_{n,\lambda}^\delta\|_{W^{m,p}(\Omega)} \le C_1 \left(\delta + r_n\right)^{\frac{k-m}{k}}.
\end{equation}
\end{theorem}

\begin{proof}
By Proposition \ref{prop_regularizationerror} and \eqref{para_choice}, we obtain
\begin{align*}
\|f - f_{n,\lambda}^\delta\|_{L^p(\Omega)} \le C \lambda^{\frac{1}{p}} c_{\sigma,\rho}^{\varphi}(f) \leq C (\delta + r_n ),
\end{align*}
where $\rho$ is any representing measure of $f$. Moreover, by Theorem \ref{sobolevembed}, we have
\begin{align*}
\|f_{n,\lambda}^\delta\|_{W^{m,p}(\Omega)} \le C c_{\sigma,\rho_{n,\lambda}^{\delta}}^{\varphi}(f_{n,\lambda}^{\delta}) \le C c_{\sigma,\rho}^{\varphi}(f),
\end{align*}
Then, interpolating between $L^p(\Omega)$ and $W^{k,p}(\Omega)$ with $0 \le \theta := m/k \le 1$ yields
\begin{align*}
\|f - f_{n,\lambda}^\delta\|_{W^{m,p}(\Omega)} \le C \|f - f_{n,\lambda}^\delta\|_{L^p(\Omega)}^{\frac{k-m}{k}} \|f - f_{n,\lambda}^\delta\|_{W^{k,p}(\Omega)}^{\frac{m}{k}},
\end{align*}
which gives the desired result in (\ref{eq_Tikerror}).
\end{proof}

\section{Numerical Experiments}
\label{se_numerics}
In this section, we present some numerical examples to validate the theoretical results established in the previous sections. Specifically, we investigate the numerical performance of the proposed Tikhonov regularization scheme penalized by the generalized Barron norm, with a focus on how the growth rate of the norm function $\varphi$ affects the simultaneous approximation of the target function and its higher-order derivatives. For classical results concerning regularization theory for ill-posed problems, we refer, for instance, to \cite{EHN96, LP13}.

\subsection{Experimental Setup}
\label{subsec:setup}
Throughout the simulations, we consider the domain to be the $d$-dimensional hypercube $\Omega = (0, 1)^d$, where the spatial dimension is chosen as $d=1$ or $d=5$. The target function $f(x)$ is a randomly generated Gaussian kernel-type function defined by:
\begin{equation}\label{eq_numtruesolution}
    f(x) = \sum_{j=1}^{T} \phi_j \exp\left(-\gamma_j |x \cdot l_j - \mu_j|^2\right),
\end{equation}
where the number of kernels is fixed to \(T = 10\). The scalar centers \(\mu_j \in (0,1)\), directions \(l_j \in \mathbb{S}^{d-1}\) (normalized in the \(\ell_2\)-norm), and coefficients \(\phi_j \in [-1,1]\) are generated randomly. The width parameters \(\gamma_j\) are sampled uniformly from \((0,30)\).

A training dataset consisting of $D$ sampling points $\{x_i\}_{i=1}^D$ is uniformly drawn from $\Omega$. To evaluate robustness against noise, the discrete measurements are corrupted by uniform relative noise:
\begin{equation}\label{eq_numnoise}
    f_\delta(x_i) = f(x_i) + \delta \xi_i, \quad i = 1, \dots, D,
\end{equation}
where $f(x_i)$ denotes the discrete measurement of the true solution (\ref{eq_numtruesolution}), the relative noise level is set to $\delta = 0.2$, and $\xi_i$ are the noise variables. For simplicity, we take $p=2$ and rescale $\xi$ so that $\sum_{i=1}^D\xi_i^2 = \sum_{i=1}^D f^2(x_i)$.

We employ a shallow neural network $f_n(x) = \sum_{i=1}^n a_i \sigma(w_i \cdot x + b_i)$ with an oscillatory activation function $\sigma(x) = \sin(15x)$ to capture different frequency characteristics. The network width is tested for $n = 500$ and $n = 1000$. The empirical Tikhonov functional to be minimized is defined as:
\begin{equation}\label{eq_numTikhonov}
    J_\lambda(f_n) = \frac{V_\Omega}{D} \sum_{i=1}^D |f_n(x_i) - f_\delta(x_i)|^2 + \lambda \left[ \sum_{i=1}^n |a_i| \varphi(\|w_i\|_1 + |b_i|) \right]^2,
\end{equation}
where $V_\Omega = 1$ is the volume of $\Omega$ and $f_\delta(x_i)$ is the measurement data in (\ref{eq_numnoise}). To ensure stable convergence, optimization is performed via a three-stage process: 

\begin{enumerate}
 \item Initial Adam phase: a relatively large learning rate of $0.05$ for $3500$ epochs, and the regularization parameter $\lambda$ in (\ref{eq_numTikhonov}) is set to
    \begin{align*}
        \lambda = \lambda_c \left(n^{\frac{(1-p-1/(d+1))}{p}}+\delta\right)^p
    \end{align*}
    in accordance with Theorem \ref{para_choice_thm}. Here $\lambda_c$ is a scaling parameter controlling the regularization strength. In our experiments, $\lambda_c=0.1$ throughout.

    \item Second Adam phase: a relatively small learning rate of $0.005$ for $3500$ epochs, and the regularization parameter $\lambda$ is updated to
    \begin{align*}
        \lambda = \lambda_c(n^{\frac{(1-p-1/(d+1))}{p}}+\delta)^p/\max\left(1,\sum_{i=1}^n |a_i^1|\varphi(\|w_i^1\|_1+|b_i^1|)\right)
    \end{align*}
    where $a^1_i, b^1_i, w_i^1$ are the parameters of the neural network after the above Initial Adam phase, and the maximum in the denominator is applied to avoid blow-up of the regularization parameter.
    
    \item L-BFGS phase: a final refinement phase using the L-BFGS algorithm trained for $500$ epochs. The choice of the regularization parameter is similar to that in the above Second Adam phase, but the neural network parameters are replaced by those obtained after Second Adam phase.
\end{enumerate}

We compare the least squares method, denoted by \texttt{no\_reg} throughout this section, against various generalized Barron regularizations where $\varphi(x)$ corresponds to the Taylor expansions of the exponential function with different growth rates, as well as the exponential function itself:
\begin{itemize}
    \item \textbf{Linear}: $\varphi(x) = 1 + x$;
    \item \textbf{Quadratic}: $\varphi(x) = 1 + x + \frac{x^2}{2}$;
    \item \textbf{Cubic}: $\varphi(x) = 1 + x + \frac{x^2}{2} + \frac{x^3}{6}$;
    \item \textbf{Quartic}: $\varphi(x) = 1 + x + \frac{x^2}{2} + \frac{x^3}{6} + \frac{x^4}{24}$;
    \item \textbf{Exponential}: $\varphi(x) = \exp(x)$.
\end{itemize}

\subsection{Numerical Results in One Dimension ($d=1$)}
\label{subsec:res_1d}

We first analyze the results for $d=1$ with a sampling dataset of size $D=2000$. Tables~\ref{tab:d1_n500} and \ref{tab:d1_n1000} present the mean relative errors in the Sobolev semi-norms for derivatives up to the $4$th order, for $n = 500$ and $n = 1000$, respectively.

\begin{table}[htbp]
    \centering
    \caption{Mean relative \( L^2 \) errors for \( d = 1 \) and \( n = 500 \) over 20 random realizations.}
    \label{tab:d1_n500}
    {\footnotesize
    \begin{tabular}{lccccc}
        \toprule
        Norm function $\varphi$ & Function value & 1st Derivative  & 2nd Derivative  & 3rd Derivative & 4th Derivative \\
        \midrule
        \texttt{no\_reg}    & 2.1406e-02 & 5.6171e-01 & 9.8439e+00 & 1.1715e+02 & 7.3824e+02 \\
        \texttt{linear}     & 1.9039e-02 & 1.3896e-01 & 5.0564e-01 & 1.9380e+00 & 7.0804e+00 \\
        \texttt{quadratic}  & 1.9901e-02 & 1.3667e-01 & 3.7390e-01 & 6.6932e-01 & 1.1062e+00 \\
        \texttt{cubic}      & 2.1188e-02 & 1.3191e-01 & 3.4000e-01 & 5.5959e-01 & 7.0888e-01 \\
        \texttt{quartic}    & 2.1991e-02 & 1.4115e-01 & 3.5113e-01 & 4.9925e-01 & 6.1360e-01 \\
        \texttt{exponential} &2.2625e-02 &1.4333e-01 &3.5709e-01 &5.0613e-01 &6.1716e-01\\
        \bottomrule
    \end{tabular}
    }
\end{table}

\begin{table}[htbp]
    \centering
    \caption{Mean relative \( L^2 \) errors for \( d = 1 \) and \( n = 1000 \) over 20 random realizations.}
    \label{tab:d1_n1000}
    {\footnotesize
    \begin{tabular}{lccccc}
        \toprule
        Norm function $\varphi$ & Function value & 1st Derivative  & 2nd Derivative  & 3rd Derivative & 4th Derivative \\
        \midrule
        \texttt{no\_reg}    & 2.1375e-02 & 4.7429e-01 & 7.6263e+00 & 8.5356e+01 & 5.4629e+02 \\
        \texttt{linear}     & 1.7061e-02 & 1.3449e-01 & 5.1508e-01 & 2.1823e+00 & 8.8114e+00 \\
        \texttt{quadratic}  & 1.7259e-02 & 1.2497e-01 & 3.4648e-01 & 5.7734e-01 & 8.0657e-01 \\
        \texttt{cubic}      & 1.6614e-02 & 1.2980e-01 & 3.5334e-01 & 5.3990e-01 & 7.0431e-01 \\
        \texttt{quartic}    & 1.9515e-02 & 1.3413e-01 & 3.4561e-01 & 5.1739e-01 & 6.5450e-01 \\
        \texttt{exponential} &1.9076e-02 &1.3394e-01 &3.4242e-01 &4.8985e-01 &5.9740e-01\\
        \bottomrule
    \end{tabular}
    }
\end{table}

A key insight revealed by Tables~\ref{tab:d1_n500} and \ref{tab:d1_n1000} concerns the optimization stability of neural networks under moderate relative noise. It is worth noting that for the least squares method, denoted by \texttt{no\_reg} in both tables, the high-order derivative approximations are not uniformly catastrophic; across the 20 independent random realizations, the neural network trained by the least squares method can occasionally produce acceptable profiles for high-order derivatives. However, standard training without regularization exhibits remarkably weak stability. Due to the lack of capacity control over the network parameters, the optimization trajectory is highly sensitive to the random seed and measurement noise, leading to severe approximation failures in several runs where high-order derivatives explode. This weak stability and high variance across seeds significantly inflate the statistical averages, resulting in large mean errors such as $738.24$ and $546.29$ for the $4$th derivative.

In stark contrast, introducing generalized Barron regularizers dramatically improves training stability. By monotonically increasing the growth rate of $\varphi$ from linear to exponential, the relative errors for higher-order derivatives exhibit a steady downward trend. For instance, in Table~\ref{tab:d1_n1000}, the mean error for the $4$th derivative drops from $8.8114$ (\texttt{linear}) to $0.6545$ (\texttt{quartic}). This confirms that a fast-growing $\varphi$ successfully establishes reliable structural regularization, acting as a low-pass filter that effectively suppresses noise-induced parameter variance across different optimization trials.

\subsection{Numerical Results in High Dimensions ($d=5$)}
\label{subsec:res_5d}

To evaluate the scalability and performance of our framework in higher dimensions, we extend the spatial dimension to $d=5$ and increase the number of sampling points to $D=10^5$. Tables~\ref{tab:d5_n500} and \ref{tab:d5_n1000} report the corresponding mean relative errors for $n=500$ and $n=1000$.

\begin{table}[htbp]
    \centering
    \caption{Mean relative \( L^2 \) errors for \( d = 5 \) and \( n = 500 \) over 20 random realizations.}
    \label{tab:d5_n500}
    {\footnotesize
    \begin{tabular}{lccccc}
        \toprule
        Norm function $\varphi$ & Function value & 1st Derivative  & 2nd Derivative  & 3rd Derivative & 4th Derivative \\
        \midrule
        \texttt{no\_reg}    & 3.8373e-02 & 2.6533e-01 & 1.0968e+00 & 5.3835e+00 & 2.2691e+01 \\
        \texttt{linear}     & 5.4756e-02 & 1.8311e-01 & 3.2378e-01 & 5.1021e-01 & 6.6185e-01 \\
        \texttt{quadratic}  & 7.0512e-02 & 2.3139e-01 & 4.0873e-01 & 6.2104e-01 & 7.6623e-01 \\
        \texttt{cubic}      & 7.7549e-02 & 2.5226e-01 & 4.4428e-01 & 6.5086e-01 & 7.9285e-01 \\
        \texttt{quartic}    & 8.0352e-02 & 2.6147e-01 & 4.5798e-01 & 6.6555e-01 & 8.0120e-01 \\
        \texttt{exponential} &8.1814e-02 &2.6835e-01 &4.7150e-01 &6.9000e-01 &8.3230e-01\\
        \bottomrule
    \end{tabular}
    }
\end{table}

\begin{table}[htbp]
    \centering
    \caption{Mean relative \( L^2 \) errors for \( d = 5 \) and \( n = 1000 \) over 20 random realizations.}
    \label{tab:d5_n1000}
    {\footnotesize 
    \begin{tabular}{lccccc}
        \toprule
        Norm function $\varphi$ & Function value & 1st Derivative  & 2nd Derivative  & 3rd Derivative & 4th Derivative \\
        \midrule
        \texttt{no\_reg}    & 4.3334e-02 &3.0633e-01 &1.0579e+00 &4.2979e+00 &1.2179e+01 \\
        \texttt{linear}     & 4.4707e-02 &1.5143e-01 &2.7660e-01 &4.6472e-01 &6.5023e-01 \\
        \texttt{quadratic}  & 6.5382e-02 &2.2018e-01 &3.9095e-01 &5.9114e-01 &7.4064e-01 \\
        \texttt{cubic}    & 7.2827e-02 &2.4162e-01 &4.3230e-01 &6.3729e-01 &7.8154e-01 \\
        \texttt{quartic}  & 7.8180e-02 &2.6036e-01 &4.6054e-01 &6.7328e-01 &8.1585e-01 \\
        \texttt{exponential} &8.0978e-02 &2.7017e-01 &4.7792e-01 &6.8681e-01 &8.3012e-01\\
        \bottomrule
    \end{tabular}
    }
\end{table}

As shown in Tables~\ref{tab:d5_n500} and \ref{tab:d5_n1000}, the instability of the least squares method persists in higher dimensions. While the unregularized network can capture function properties reasonably well, its variance across random seeds heavily affects the high-order derivatives, leading to accumulated mean errors (e.g., $22.691$ for the $4$th derivative). The introduction of our generalized Barron regularizers provides crucial stabilization, keeping all derivative errors within a stable tight bound ($< 0.84$) across all seeds.

Interestingly, an intriguing trade-off distinct from the 1D case emerges in the 5D setting: the relative errors across all derivative orders increase slightly as the growth rate of $\varphi$ moves from linear to quartic. For example, in Table~\ref{tab:d5_n1000}, the $4$th derivative error changes from $0.65$ (\texttt{linear}) to $0.83$ (\texttt{exponential}). This behavior reveals that under high dimension, the function landscape may become exceptionally complex. A fast-growing norm function such as \texttt{exponential} may impose overly rigid parameter bounds, which somewhat limits the immediate expressive capacity of a finite-width network. Consequently, while a smoother norm function provides reliable defense against derivative divergence, it induces a minor underfitting effect. This suggests that for high-dimensional problems, a milder growth rate (e.g., \texttt{linear} or \texttt{quadratic}) of the norm function, or a significantly larger network width combined with more training epochs, is preferable to optimize the trade-off between smoothing stabilization and expressive power.

\section{Conclusion}\label{se_conclusion}
In this work, we proposed a novel family of generalized Barron spaces \(B_\sigma^\varphi\) for shallow neural networks by coupling each activation \(\sigma\) with a norm function \(\varphi\). This generalization remedies limitations of previous formulations, which may degenerate for non-homogeneous \(\sigma\), by producing a well-defined Banach space for broad classes of activation functions. We proved that, under appropriate conditions on \(\varphi\) and \(\sigma\), the space \(B_\sigma^\varphi\) continuously embeds into \(W^{k,\infty}(\Omega)\), meaning that the \(\varphi\)-weighted norm controls derivatives up to order \(k\). Thus, our results link parametric complexity, measured by \(\|f\|_{B_\sigma^\varphi}\), to the smoothness of the realized function.
    
From an approximation perspective, we obtained dimension-free error bounds for approximating functions in \(B_\sigma^\varphi\) by finite-width neural networks. A Monte Carlo sampling argument yields the standard \(O(n^{1/p-1})\) rate in type-\(p\) target spaces. Under stronger assumptions on \(\varphi\) and \(\sigma\), we derived faster rates that recover and extend known results for classical and spectral Barron spaces.
    
Finally, we showed that the explicit \(\varphi\)-weighted representation cost offers a natural regularization tool: it can be computed from network parameters and controls the Sobolev norm of the output. Using this cost as a Tikhonov penalty therefore suggests regularization schemes tailored to the activation, beyond generic Sobolev norm penalties.
    
In summary, our generalized Barron space framework achieves finer-grained control and understanding of shallow neural networks: it unifies known theories, establishes novel embedding and approximation results, and provides practical insights for network design. In future work, it will be important to empirically validate these findings in more complex settings. For example, one could perform numerical experiments on forward and inverse problems for PDEs to confirm the predicted approximation rates and to compare $\varphi$-weighted norm regularization against standard methods. It would also be interesting to automate the choice of the norm function $\varphi$ via a data-driven approach and to extend the $\varphi$-framework to deep or residual network architectures. Finally, we anticipate that analyzing optimization dynamics, e.g. gradient flow, in relation to the induced $\varphi$-weighted norm could yield further insights into learning behavior under network-specific norms.

\appendix    
\section{Proof of Lemma \ref{kernel_prop}}\label{appendix_B}
    \begin{proof}[Proof of Lemma \ref{kernel_prop}]
For a measurable rectangle \(A\times B\), we first define
\[
\eta(A\times B):=\int_A K(x,B)\,\mu(\mathrm{d}x).
\]
The assumption \(\int_X|K|(x,Y)\,\mathrm{d}|\mu|(x)<\infty\) implies that this set function extends uniquely to a finite signed complex measure \(\eta\) on \((X\times Y,\mathscr{A}\otimes\mathscr{B})\). Its total variation satisfies
\[
|\eta|(A\times B)\le \int_A |K|(x,B)\,\mathrm{d}|\mu|(x).
\]
Define \(\nu\) as the \(Y\)-marginal of \(\eta\), namely \(\nu(B):=\eta(X\times B)\). Then \(\nu\) is a finite signed complex measure and
\[
|\nu|(B)\le |\eta|(X\times B)
\le \int_X |K|(x,B)\,\mathrm{d}|\mu|(x).
\]

For a measurable function \(f:Y\to\mathbb{C}\) satisfying the integrability assumption in the lemma, the function \((x,y)\mapsto f(y)\) belongs to \(L^1(|\eta|)\). Fubini's theorem for finite signed complex measures therefore gives
\[
\int_Y f(y)\,\nu(\mathrm{d}y)
=\int_{X\times Y}f(y)\,\eta(\mathrm{d}x,\mathrm{d}y)
=\int_X\left(\int_Yf(y)\,K(x,\mathrm{d}y)\right)\mu(\mathrm{d}x).
\]
The same theorem also yields the asserted measurability and integrability of the inner integral.
    \end{proof}

\section{Type-$\min\{2,p\}$ spaces}\label{se_appendix}
    \begin{proposition}\label{sobolev_is_typep}
       For \( k \in \mathbb{N} \) and \( p \in (1,\infty) \), the Sobolev space $W^{k, p}(\Omega)$ is type-$\min\{2,p\}$.
    \end{proposition}
    \begin{proof}
        Let $f_1,\dots,f_n\in W^{k,p}$ and $\varepsilon_1,\dots,\varepsilon_n$ be random variables obeying Rademacher distribution. Then, there holds
        \begin{align*}
    \mathbb{E}\left \Vert\sum_{j=1}^n \varepsilon_jf_j \right\Vert_{W^{k,p}}^p &=\mathbb{E}\left(\sum_{\vert \alpha\vert\le k}\int_{\Omega}\left \vert\sum_{j=1}^n\varepsilon_j\partial^\alpha f_j(x) \right \vert^p \mathrm{d}x \right)\\
    &= \sum_{\vert \alpha\vert\le k}\int_{\Omega}\mathbb{E}\left(\left \vert\sum_{j=1}^n\varepsilon_j\partial^\alpha f_j(x) \right \vert^p \right)\mathrm{d}x\\
    &\le b_p^p\sum_{\vert \alpha\vert \le k}\int_{\Omega} \left(\sum_j\vert \partial^\alpha f_j\vert^2 \right)^{\frac{p}{2}}\mathrm{d}x,
\end{align*}
    and the last inequality is derived from the Khintchine inequality, in which $b_p$ is a constant in the Khintchine inequality that depends on $p$.

    For the case $p\ge 2$,
\begin{align*}
    \sum_{\vert \alpha\vert \le k}\int_{\Omega} \left(\sum_j\vert \partial^\alpha f_j\vert^2 \right)^{\frac{p}{2}} &= \sum_{\vert \alpha\vert \le k} \left \Vert\sum_j\vert \partial^\alpha f_j\vert^2
     \right \Vert_{L^{\frac{p}{2}}(\Omega)}^{\frac{p}{2}}\\
     &\le\sum_{\vert \alpha\vert\le k}\left(\sum_j\left \Vert(\partial^\alpha f_j)^2 \right \Vert_{L^{\frac{p}{2}}(\Omega)} \right)^{\frac{p}{2}}\\
     & =\sum_{\vert \alpha\vert\le k}\left\{\sum_j\left[\int_{\Omega}\vert\partial^\alpha f_j\vert^p \right]^{\frac{2}{p}} \right\}^{\frac{p}{2}}\\
     & \le \sum_{\vert \alpha\vert\le k}\left\{\sum_j\left[\sum_{\vert \alpha^{'}\vert\le k}\int_{\Omega}\vert\partial^{\alpha^{'}} f_j\vert^p \right]^{\frac{2}{p}} \right\}^{\frac{p}{2}}\\
     &=d^k\left(\sum_j\left \Vert f_j\right \Vert_{W^{k,p}}^{2} \right)^{\frac{p}{2}}.
\end{align*}
    Therefore,
    \begin{equation*}
        \left (\mathbb{E}\left \Vert\sum_{j}\varepsilon_jf_j \right\Vert_{W^{k,p}}^p\right )^{\frac{2}{p}}\le C\sum_j \Vert f_j\Vert_{W^{k,p}}^2.
    \end{equation*}
    Because $p\ge 2$, by H\"{o}lder inequality 
    \begin{equation*}
        \left (\mathbb{E}\left \Vert\sum_{j}\varepsilon_jf_j \right\Vert_{W^{k,p}}^2\right )^{\frac{1}{2}}\le\left (\mathbb{E}\left \Vert\sum_{j}\varepsilon_jf_j \right\Vert_{W^{k,p}}^p\right )^{\frac{1}{p}},
    \end{equation*}
    hence
    \begin{equation*}
        \mathbb{E}\left \Vert\sum_{j}\varepsilon_jf_j \right\Vert_{W^{k,p}}^2\le C\sum_j \Vert f_j\Vert_{W^{k,p}}^2.
    \end{equation*}
    For the case $p<2$, by the Jensen's inequality, we derive 
    \begin{align*}
        \left(\sum_j\vert \partial^\alpha f_j\vert^2 \right)^{\frac{p}{2}}\le \sum_j\left( \vert \partial^\alpha f_j\vert^2\right)^{\frac{p}{2}}=\sum_j\vert \partial^\alpha f_j\vert^{p},
    \end{align*}
    hence,
    \begin{align*}
        \sum_{\vert \alpha\vert \le k}\int_{\Omega} \left(\sum_j\vert \partial^\alpha f_j\vert^2 \right)^{\frac{p}{2}} &\le \sum_{\vert \alpha\vert \le k}\int_{\Omega} \sum_j\vert \partial^\alpha f_j\vert^p\\
        &=\sum_j\Vert f_j\Vert_{W^{k,p}}^p.
    \end{align*}
    Therefore,
    \begin{align*}
        \mathbb{E}\left \Vert\sum_{j}\varepsilon_jf_j \right\Vert_{W^{k,p}}^p\le C\sum_j \Vert f_j\Vert_{W^{k,p}}^p.
    \end{align*}
    This completes the proof that $W^{k,p}(\Omega)$ is type-$\min\{2,p\}$.
    \end{proof}

\section*{Acknowledgements}
This work was supported by National Key Research and Development Programs of China (No. 2023YFA1009103), NSFC (No. 92570106) and Science and Technology Commission of Shanghai Municipality (No. 23JC1400501). 

\section*{Data Availability Statement}
Data sharing not applicable to this article as no datasets were generated or analysed during the current study.

\section*{Conflict of Interest}
The authors have no conflicts of interest to declare that are relevant to the content of this article.

\normalem
\bibliographystyle{plain}
\bibliography{references}
\end{document}